\documentclass[]{siamart171218}

\usepackage{lipsum}
\usepackage{amsfonts}
\usepackage{graphicx}
\usepackage{epstopdf}
\usepackage{algorithmic}
\usepackage{tikz}
\usepackage[caption=false]{subfig}
\ifpdf
  \DeclareGraphicsExtensions{.eps,.pdf,.png,.jpg}
\else
  \DeclareGraphicsExtensions{.eps}
\fi

\newsiamremark{remark}{Remark}
\newsiamremark{hypothesis}{Hypothesis}
\crefname{hypothesis}{Hypothesis}{Hypotheses}
\newsiamthm{claim}{Claim}

\headers{ Domain Decomposition Algorithm For Eikonal Equations}{Lindsay Martin and Richard Tsai}

\title{A Multiscale Domain Decomposition Algorithm For Boundary Value Problems For Eikonal Equations\thanks{Submitted: 05/11/2018.
\funding{This research is partially supported by National Science Foundation Grants DMS-1620396 and DMS-1720171.}}}

\author{Lindsay Martin\thanks{Department of Mathematics, The University of Texas at Austin, Austin, TX
  (\email{lmartin@math.utexas.edu}).}
\and Richard Tsai \thanks{Department of Mathematics and Institute for Computational Engineering and Sciences (ICES), The University of Texas at Austin, Austin, TX,
\mbox{KTH Royal Institute of Technology, Sweden}
  (\email{\mbox{ytsai@math.utexas.edu})}.}
  }

\usepackage{amsopn}

\ifpdf
\hypersetup{
  pdftitle={A Multiscale Domain Decomposition Algorithm For Boundary Value Problems For Eikonal Equations},
  pdfauthor={ Lindsay Martin and Richard Tsai}
}
\fi

\begin{document}

\maketitle

\begin{abstract}
   In this paper, we present a new multiscale domain decomposition algorithm for computing solutions of static Eikonal equations. The new method is an iterative two-scale method that uses a parareal-like update scheme in combination with standard Eikonal solvers. The purpose of the two scales is to accelerate convergence and maintain accuracy. We adapt a weighted version of the parareal method for stability, and the optimal weights are studied via a model problem. Numerical examples are given to demonstrate the method.
\end{abstract}



\section{Introduction}\label{sec:intro}
The Eikonal equation has many applications in optimal control, path planning, seismology, geometrical optics, etc. The equation is fully nonlinear and classified as a Hamilton-Jacobi equation. Usually,  classical solutions do not exist, and the unique viscosity solution is sought after.
Our goal is to numerically solve the following boundary value problem for the static Eikonal equation:

\begin{align}\label{eq:eikonal}|\nabla u(x)|&=r_{\epsilon}(x), \; x \in \Omega \subset \mathbb{R}^d \\ u(x)&=g(x),\; x\in \Gamma \subset \partial \Omega \end{align}

In particular, we are interested in the case where $$r_{\epsilon}(x)=r_0(x)+a_\epsilon(x),$$ where $r_0$ is smooth and $a_\epsilon$ describes multiscale features. 
Many serial algorithms exist for computing numerical solutions to Eikonal equations. However, these algorithms have limitations when applied to large scale discretized systems.  Since we are interested in Eikonal equations that have multiscale features, a very fine grid discretization is needed in order to accurately capture the fine scale features. This creates a large system of coupled nonlinear equations to solve. Therefore, the numerical solutions are expensive to compute and speed up is desired. The most popular serial algorithms are the Fast Sweeping Method (FSM) \cite{tsaicheng03, zhao05} and the Fast Marching Method (FMM) \cite{tsit195,seth96} which have complexity $\mathcal{O}(N)$ or $\mathcal{O}(N\log N)$, respectively. Here, $N$ is the total number of grid points. Hidden in the $\mathcal{O}(N)$ complexity of FSM is a constant that corresponds to the number of times a characteristic curve of \cref{eq:eikonal} ``turn around.''

There are several approaches to reducing the computational cost of numerically solving Eikonal equations. For certain periodic functions $r_\epsilon$, one approach is homogenization \cite{Luo11, Ober09}. The goal of homogenization is to derive an effective function, $\overline{r}$, that accurately describes the effective properties of $r_\epsilon$ in the solution. Once $\overline{r}$ is known, the homogenized equation can be solved on the coarse grid which is independent of the small parameter $\epsilon$.  For more general $r_\epsilon$, we consider domain decomposition methods. The development of domain decomposition algorithms for Eikonal equations is nontrival  because of the causal nature of the equations. Standard domain decomposition methods can be difficult to apply because information may not be known at the boundaries of subdomains a priori. Furthermore, the causal relations among the subdomains may change depending on the solutions. Our new algorithm combines features from parareal methods and standard Eikonal solvers in order achieve speed up and maintain accuracy. A set of coarse grids is used to set up boundary conditions for the subdomains. After each subdomain is processed in parallel, the method uses are parareal-like update in order to speed up the accuracy of the solution on the coarse grids. 

Next we give an overview of the discretization of \cref{eq:eikonal} and Fast Sweeping Methods, followed by a review of current parallel methods for Eikonal equations. The paper is organized as follows. In  \cref{sec:para}, we give an overview of parareal methods. Our new algorithm is presented in \cref{sec:main}. The stability analysis, complexity and speed up are given in \cref{sec:analysis}, experimental
results are in \cref{sec:results}, and the summary and conclusion follow in \cref{sec:conclusions}.

\subsection{Upwind discretization and FSM}\label{sec:FSM}

 The Eikonal equation \cref{eq:eikonal} can be derived from an optimal control problem. Suppose a particle travels at speed ${F:\Omega \to \mathbb{R}}$ and its direction of travel is the control of the system. Let $g:\partial \Omega \to \mathbb{R}$ be the exit-time penalty charged at the boundary. Then the value function $u(x)$ is defined to the minimum time it takes to travel from $x$ to $\partial \Omega$. In \cite{cranlions83}, it is shown the viscosity solution to \cref{eq:eikonal} coincides with the value function of the optimal control problem and the characteristics of the PDE coincide with the optimal paths for moving through $\Omega$. 
 
 In our case $F(x)=1/r_\epsilon(x)$. Thus, we refer to $r_\epsilon$ as the slowness function. For this paper, we choose the following first-order upwind discretization on a uniform Cartesian grid. Let $u_{i,j}$ denote the numerical solution at $\mathbf{x}_{i,j}$. For the sake of notation, we will omit the numerical solution's dependence on the grid size $h$. We use a Godonuv upwind scheme to discretize the Eikonal equation at points in the interior of the computational domain \cite{rouytour92}:  \begin{equation}\label{eq:godonuv}\sqrt{\max(a^+,b^-)^2+\max(c^+,d^-)^2}=r_{i,j}, \end{equation} where \begin{align*} a & = D_x^-u_{i,j} = \frac{u_{i,j}-u_{i-1,j}}{h}\\ b & = D_x^+u_{i,j} = \frac{u_{i+1,j}-u_{i,j}}{h} \\  c & = D_y^-u_{i,j} = \frac{u_{i,j}-u_{i,j-1}}{h}\\  d & = D_y^+u_{i,j} = \frac{u_{i,j+1}-u_{i,j}}{h}, \end{align*} for $i=1,\ldots I-1$ and $j=1,\ldots,J-1$. Here, we have  $x^+=\max(x,0)$ and ${x^-=\max(-x,0)}$. On the boundary nodes, we will use a one sided difference, i.e., in \cref{eq:godonuv} use $b^-$ in place of $ \max(a^+,b^-)$ if $i=0$, $a^+$ in place of  $\max(a^+,b^-)$ if $i=I$, $d^-$ in place of $\max(c^+,d^-)$ if $j=0$, and $c^+$ in place of $\max(c^+,d^-)$ if $j=J$. 
 
This discretization is consistent and monotone and converges to the viscosity solution as $h \to 0$ \cite{barsoug91}. The upwind scheme is also causal, i.e., $u_{i,j}$ depends only on the neighboring grid values that are smaller. After discretization, we have a system of $N=(I+1)(J+1)$ coupled nonlinear equations. A simple approach is to solve the system iteratively \cite{rouytour92}. However, it is important to take advantage of the causality of the solution. In the fast marching method (FMM) \cite{tsit195,seth96}, the solution is updated one grid node at a time and the ordering of grid nodes is given by whichever grid node has the smallest value at the time of updating. Because a heapsort algorithm is needed, the complexity is $\mathcal{O}(N \log N)$. Next we describe the fast sweeping method (FSM)  \cite{tsaicheng03, zhao05} which we have chosen to use in our method. FSM uses Gauss-Seidel updates following a predetermined set of grid node orderings. For simplicity, we will describe the algorithm in two dimensions. 
\subsection*{Initialization} Set $u_{i,j}=g_{i,j}$ for $x_{i,j}$ on or near the computational boundary. These values are fixed in later iterations. For the all other grid nodes, assign a large positive value.
\subsection*{Sweeping iterations} A compact way of writing the grid orderings in \\ C/C++ is: 

\texttt{for(s1=-1;s1<=1;s1+=2)}

\texttt{for(s2=-1;s2<=1;s2+=2)}

\texttt{for(i=(s1<0?I:0);(s1<0?i>=0:i<=I);i+=s1)}

\texttt{for(j=(s2<0?J:0);(s2<0?j>=0:j<=J);j+=s2)}

\subsection*{Update formula:}
For each grid node $\mathbf{x}_{i,j}$ whose value is not fixed during the initialization, compute the solution to \cref{eq:godonuv} using the current values at the neighboring grid nodes. Denote the solution by $\tilde u$, then the update formula is as follows: \begin{equation}\label{eq:FSMupdate} u_{i,j}^{new}=\min(u_{i,j}^{curr},\tilde u)\end{equation}

The alternating ordering of sweeping ensures that all the directions of characteristics are captured. In \cite{zhao05} it is shown that with the first order Godonuv upwind scheme, $2^d$ sweeps is sufficient to compute the numerical solution to first order in $h$. The exact number of sweeps needed is related to the number of times characteristics change directions. Thus in general the computational complexity of the fast sweeping method is $O(N)$ with the caveat that the constant in front of $N$ can be very large depending on the characteristics of the equation.  In \cref{sec:main}, we will describe how we use the fast sweeping method as the Eikonal equation solver in our method.

 \subsection{Review of current parallel methods} Here we give a brief overview of existing parallel approaches. In \cite{zhao07}, the author proposes two parallelizations of FSM. The first performs the $2^d$  sweeps of the domain on different processors and after each iteration information is shared by taking the minimum value at each grid node from each sweep.  The second is a domain decomposition method that performs FSM on each subdomain in parallel. The information is shared along mutual boundaries after each iteration. The drawbacks to this method are that subdomains have to wait to be updated until the information propagates to that part of the domain and the number of sweeping iterations may be more than the number needed in serial FSM.

In \cite{detrix13} the authors introduce a method that takes advantage of the following fact: for the upwind scheme \cref{eq:godonuv}, certain slices of the grid nodes do not directly depend on each other. The method uses FSM where the sweeping ordering is designed to allow these sets of grid nodes to be updated simultaneously. The advantage of this method is that the number of iterations needed in the parallel implementation is equal to the serial FSM.

Several algorithms have been developed to parallelize FMM. In \cite{Breuss11}, the authors propose a domain decomposition method for FMM. The main idea is to split the boundary among different processors which leads to an equation dependent method. In \cite{Yang17}, another domain decomposition algorithm is presented for FMM. In this method, the computational domain is split among different processors and a novel restarted narrow band approach which coordinates the communications among the boundaries of the domains is used. 

Domain decomposition methods that utilize two scales can be found in \cite{Cac12, chacvlad15}. In \cite{Cac12}, the authors take advantage of the optimal control formulation of Eikonal equations. First, the algorithm computes the solution of \cref{eq:eikonal} on a coarse grid. Next, the domain decomposition is determined by the feedback optimal control. Lastly, the solution of the equation is computed on a fine grid in each subdomain. However, the algorithm can lead to complex division of the domain. The method in \cite{chacvlad15} is a parallelization of the Heap Cell method (HCM) \cite{chacvlad12}.  HCM maintains a list of cells to be processed. The order of processing is determined by an assigned cell value that is given by an estimate of the likelihood that that cell influences other cells.  If it is determined that a cell highly influences other cells it should be processed first. The method mimics FMM on the coarse level, and FSM is used at the cell level. The parallelization of HCM divides the cells evenly among $p$ heaps and performs HCM among each individual heap. If a cell is tagged for reprocessing then it is added to the heap with the current lowest number of cells. This method was found to achieve the best speed up on problems where the amount of work per cell is high.

\section{Overview of parareal methods}
\label{sec:para}

Parareal methods \cite{Lions01,Bal02} were developed to parallelize numerical computations of the solutions to ODEs of the form \begin{equation}\label{eq:ODE} \frac{d}{dt}u=f(u), \; u(0)=u_0 \end{equation}  on bounded time interval $[0,T]$. Let $u_n^k$ be the computed solution at iteration $k$ at time $t_n=nH$. Let $C_H$ and $F_H$ be the numerical coarse and fine integrators, over time step $H$. The idea is that $C_H$ is less accurate and inexpensive to compute, and $F_H$ is very accurate and expensive to compute. The parareal update scheme is then defined as \begin{equation}\label{eq:para} u_{n+1}^{k+1}=C_H(u_n^{k+1})+F_H(u_n^k)-C_H(u_n^k), \; n,k=0,1,2,\ldots N\end{equation} with initial conditions \begin{equation}\label{eq:ICpara} u_0^k=u_0, \; k=0,1,2,\ldots N\end{equation} The zeroth iteration is given by \begin{equation}\label{eq:zeroiter} u_{n+1}^0=C_H(u_n^0), \; n=0,1,2,\ldots N \end{equation} 

The integrations $F_H(u_n^k)$ are independent for each $n$ and  can be computed in parallel. If $C_H$ is of order 1, then under certain assumptions, the error after $k$ iterations of the parareal scheme is of order $o(H^{k}+e^f)$ where $e^f$ is the global error from solving \cref{eq:ODE} with the fine integrator $F_H$ \cite{Maday10} . The method only provides speed up if $k$ is much smaller than $N$. 

The method is generally unstable for hyperbolic problems and problems with imaginary eigenvalues \cite{Staff05,Bal05}. Parareal methods for highly oscillatory ODEs can be found in \cite{Ariel16, Haut14}. In \cite{Gander14}, analysis of the parareal method on a class of ODEs originating in Hamiltonian dynamical systems is presented, and in \cite{ Legoll13} the parareal method is applied to stiff dissipative ODEs. Recently, a ``weighted" parareal scheme, called $\theta$-parareal, was proposed in \cite{Ariel17}. 
Following the scheme in \cite{Ariel17}, let \begin{equation}\label{eq:wtdpara} u_{n+1}^{k+1}=\theta C_H(u_n^{k+1})+ (1-\theta)C_H(u_n^k)+F_H(u_n^k)-C_H(u_n^k)\end{equation} which simplifies to \begin{equation}\label{eq:wtdparasimp} u_{n+1}^{k+1}=\theta C_H(u_n^{k+1}) +F_H(u_n^k)-\theta C_H(u_n^k).\end{equation} In \cite{Ariel17}, the ``weight" $\theta$ is generalized to an operator which maps $C_Hu$ to a small neighborhood of $F_Hu$. In this paper, we only let $\theta$ be a real number which may vary for each grid node, i.e., $\theta=\theta_n^k$.

Several properties of the parareal scheme are appealing when solving Eikonal equations. 
\begin{itemize}
\item Parareal methods use communications between the two-scales in order to propagate information quickly through time. Because the fine integrations can be computed in parallel, the method is able to deal with a large number of unknowns. 
\item The characteristics of Eikonal equations also have a ``time-like" structure which makes parareal methods attractive. 
\end{itemize}
The main challenge in applying the parareal scheme to Eikonal equations is that we are now dealing with an infinite number of characteristics simultaneously. We also must be able to handle the collision of characteristics which should be captured accurately in the numerical solution in order to compute the viscosity solution. We adapt the $\theta$-parareal scheme in order to stabilize the new method.

\section{The new method}\label{sec:main}

The method is a domain decomposition method that uses two scales to resolve the fine scale features in $r_{\epsilon}$ and propagate information through the computational domain. We use FSM as the Eikonal equation solver on the coarse and fine grid. An adapted version of the $\theta$-parareal method is used to propagate information along the characteristics efficiently where the weight $\theta$ stabilizes the method. The optimal choice of weights for stability is studied in  \cref{sec:analysis}. First, we will demonstrate the method on a one dimensional problem and then explain how to set up the method in two dimensions which can be generalized to higher dimensions.

\subsection{One dimensional example} Consider the following one dimensional \\ Eikonal equation: \begin{align}\label{eq:1d} |u_x| & =r(x), \; 0<x<1, \\ \label{eq:bc1d} u(0) & =u(1)=0, \end{align} where \begin{equation} r_\epsilon(x)= 1 + 10 e^{\frac{(x-.75)^2}{2(.01)^2}}.\end{equation} \cref{fig:1Dexamples} shows the plot of the slowness function $r_\epsilon$. Let the coarse grid be defined by $$\Omega^H:=\{jH: j = 0,1, \ldots,N\},$$ where $H=1/N$ and for $i=0,1,\ldots, N-1.$ Define the fine grids by $$\Omega^h_i := \{iH+mh: m=0,1,\ldots M\},$$ where $h=1/(MN).$ Define $\Omega^h:=\bigcup_{i=0}^{N-1}\Omega^h_i$. The solution to the upwind Godonuv scheme in one dimension is given by \begin{equation}\label{eq:1dgod} C_H(U_{i-1},U_{i+1}):=\min(U_{i-1},U_{i+1})+r(X_i)H.\end{equation}

 We see that if we only solve \cref{eq:1d} on the coarse grid, the bump in the slowness function is not seen and the solution is very inaccurate. There are also points in $\Omega^H$ where the flow of characteristics is incorrect. Therefore, we keep track of wind direction, i.e., which neighboring grid node gives the minimum in \cref{eq:1dgod} . Let $X_i=iH$. We denote the numerical solution at the $k$th iteration at the coarse grid  node $X_i$ by $U_i^k$. For grid nodes on the subintervals, $\Omega^h_i$, let $X_{i_m}=iH+mh$ and $u^k_{i_m}$ be the numerical solution at the $k$th iteration at the fine grid node $X_{i_m}$.  For each coarse grid node , $X_i,$  $i=1,\ldots, N-1$, we will get two values from the fine grid computations. One value is from $\Omega_{i-1}^h$ and another from $\Omega_{i}^h$. Let $u_i^k$ be the fine grid solution at the $k$th iteration at $X_i$ which we will define in step 3.
\begin{figure}[h]
\centering
\includegraphics{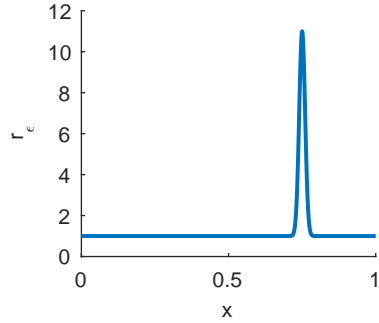}
\caption{$ r_\epsilon(x)= 1 + 10 e^{\frac{(x-.75)^2}{2(.01)^2}}$.  }
\label{fig:1Dexamples}
\end{figure}
The method is as follows:

\subsection*{Step 1: Initialization}  Solve  \cref{eq:1d} with boundary conditions \cref{eq:bc1d} via FSM on the coarse grid $\Omega^H,$ and denote the solution $U^0$. If the left hand neighboring grid node is used to compute $U_i^0$, denote the wind direction at $X_i$ by $W_i^0=1$. If the right hand neighboring grid node is used, define $W_i^0=-1$.  

\subsection*{Step 2: Update boundary conditions for the subintervals} Once the coarse grid has been initialized, we use the coarse grid values, $U^k$, as boundary values for $\Omega_i^h$. The characteristics may flow into or out of $\Omega_i^h$. Thus, when setting the boundary conditions, we check the wind direction to see if the coarse grid value should be used as a boundary value.  Intuitively, if a characteristic at a coarse grid node, $x_{i_0}$ or $x_{i_M}$, is arriving into the subinterval, then we set the boundary value to $U_i$ or $U_{i+1}$ at $x_{i_0}$ or $x_{i_M},$ respectively. Otherwise, we set the boundary value to be $\infty$.

\subsection*{Step 3: Solve for $u^k$ in parallel} In parallel for each $i=0,1,\ldots, N-1$, we solve via FSM on $\Omega_i^h$  \begin{equation}\label{eq:1dsubint} |u_x|  =r(x), \; x\in [iH,(i+1)H] \end{equation} with the boundary conditions described in step 2 . Denote the solutions after sweeping by $u_{i_m}^k$ for $m=0,\ldots,M$. We keep track of the fine wind directions, $w_{i_m}^k$, in the same manner as in step 1. For each coarse grid node , $X_i,$  $i=1,\ldots, N-1$, we will get two values from the fine grid computations. One value is from $\Omega_{i-1}^h$ and another from $\Omega_{i}^h$. Consider a coarse grid point, $X_i$:
\begin{itemize}
\item If $w_{{i-1}_M}^k=w_{i_0}^k=1$, then we choose $u_i^k$ to be $u_{{i-1}_M}^k$ since the wind is flowing from left to right.
\item If $w_{{i-1}_M}^k=w_{i_0}^k=-1$,  then we choose $u_i^k$ to be $u_{{i}_0}^k$ since the wind is flowing from right to left.
\item Otherwise, we take the minimum of $u_{{i-1}_M}^k$ and $u_{{i}_0}^k$. 
\item We set $w_i^k$ to be the wind value corresponding to the fine grid point used to define $u_i^k$. \end{itemize} For the given example, $U^0$ is plotted in \cref{fig:init1d} and $u^0$ is plotted in \cref{fig:fineinit1d} .

\begin{figure}[h]
\centering
\subfloat[]{\label{fig:init1d}\includegraphics{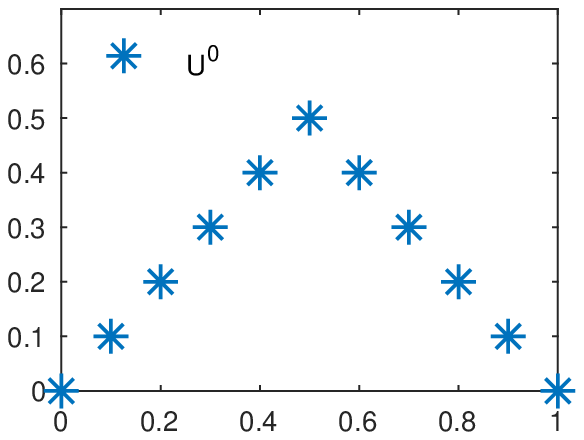}}
\subfloat[]{\label{fig:fineinit1d}\includegraphics{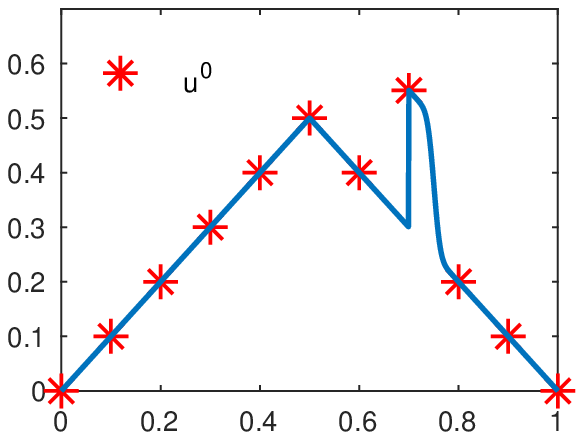}}
\caption{\protect\subref{fig:init1d} Plot of $U^0$. \protect\subref{fig:fineinit1d} Plot of $u^0$ using $U^0$ as boundary conditions as defined in step 2.}
\label{fig:iterations01d}
\end{figure}
\subsection*{Step 4: Coarse grid updates} Now we compute $U^{k+1}$. We use the previous coarse and fine wind directions to determine whether or not we will use a weighted correction. We sweep the grid as in FSM and the update formula is as follows:
\begin{itemize}
\item Let $\tilde{U}=C_H(U_{i-1}^{k+1},U_{i+1}^{k+1})$. If the left hand neighboring grid node was used to compute $\tilde{U}$ then denote the current wind direction  $\tilde W=-1$. If the right hand neighboring grid node was used, define $\tilde W=1$.
\item If $W_i^k=w_i^k=\tilde W$ then we use a weighted correction update,i.e., \begin{equation} U_i=\theta \tilde U+u_i^k-\theta C_H(U_{i-1}^k,U_{i+1}^k)\end{equation} and $W_i=w_i^k$.
\item Otherwise we set $U_i=u_i^k$ and $W_i=w_i^k$.
\item After the weighted corrections, the solutions may have the wrong causality because information on the fine grid that was not seen previously has now been propagated to the coarse level. To correct this, we implement a causal sweep after each coarse grid update. Sweeping the coarse grid in both directions, the causal update is as follows:
\begin{itemize}
\item If $W_i=1$ and $U_i< U_{i-1}$, then $U_i=U_{i-1}$.
\item If $W_i=-1$ and $U_i< U_{i+1}$, then $U_i=U_{i+1}$.
\end{itemize}
\end{itemize}
After the causal sweep on the coarse grid, denote the solution by $U^{k+1}$ and the wind directions by $W^{k+1}$. Repeat steps 2-4 until convergence.

In \cref{fig:iter11d}, we see that at $X_7=0.7$ the effect of the Gaussian bump in $r_\epsilon$ has been propagated to the coarse  level. Before the causal sweep, $U_6<U_7$, but $W_6=-1$. Therefore, after the casual sweep, $U^1_6=U^1_7$.  \cref{fig:iterations21d,fig:iterations31d} show the next two iterations of the method which converges at $k=3$.
Next, we introduce the method in two dimensions in more detail.
\begin{figure}[h]
\centering
\subfloat[]{\label{fig:iter11d}\includegraphics{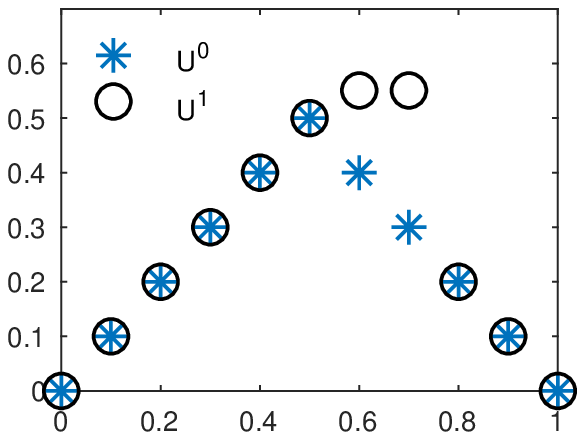}}
\subfloat[]{\label{fig:iterfine11d}\includegraphics{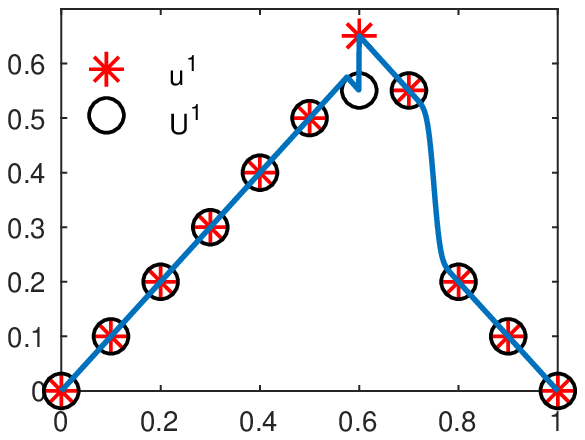}}
\caption{\protect\subref{fig:iter11d} Plot of  $U^0$ and $U^1$. \protect\subref{fig:iterfine11d} Plot of $u^1$ using $U^1$ as boundary conditions as defined in step 2.}
\label{fig:iterations11d}
\end{figure}
\begin{figure}[h]
\centering
\subfloat[]{\label{fig:iter21d}\includegraphics{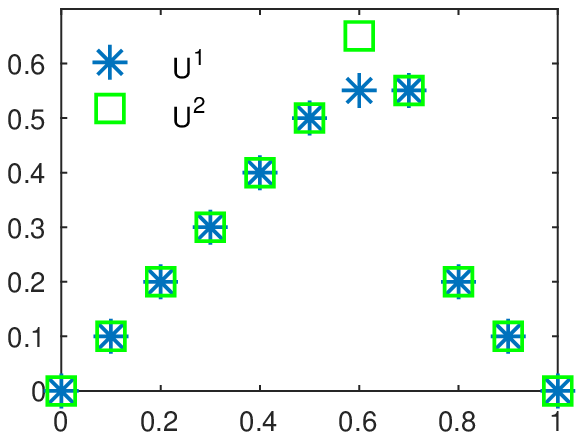}}
\subfloat[]{\label{fig:iterfine21d}\includegraphics{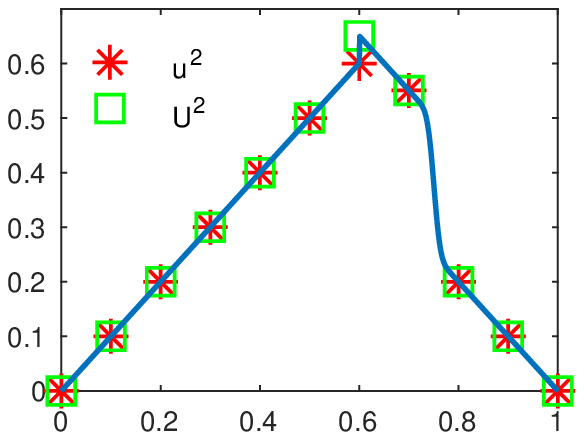}}
\caption{\protect\subref{fig:iter21d} Plot of  $U^1$ and $U^2$. \protect\subref{fig:iterfine21d} Plot of $u^2$ using $U^2$ as boundary conditions as defined in step 2.}
\label{fig:iterations21d}
\end{figure}
\begin{figure}[h]
\centering
\subfloat[]{\label{fig:iter31d}\includegraphics{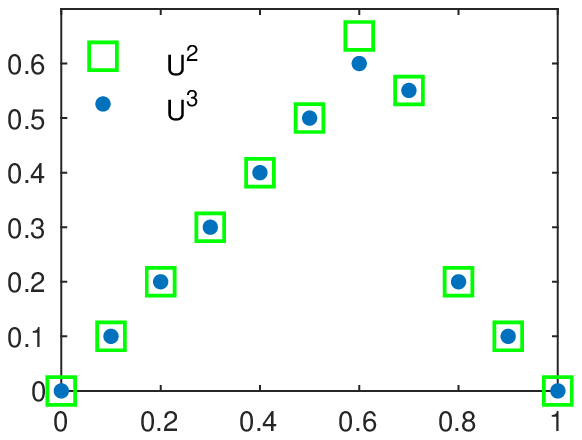}}
\subfloat[]{\label{fig:iterfine31d}\includegraphics{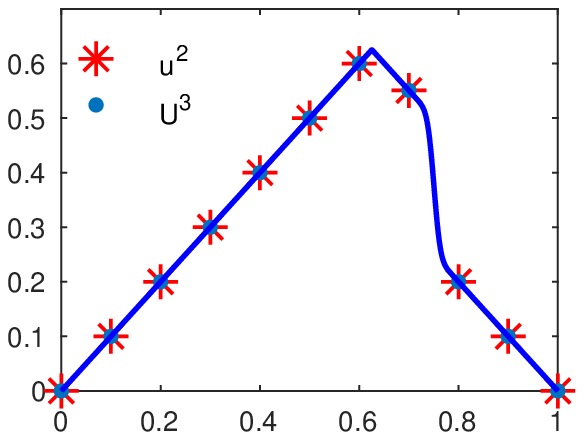}}
\caption{\protect\subref{fig:iter31d} Plot of  $U^2$ and $U^3$. \protect\subref{fig:iterfine31d} Plot of $u^3$ using $U^3$ as boundary conditions as defined in step 2. We see that the method has converged and $U^3=u^3$.}
\label{fig:iterations31d}
\end{figure}

\subsection{New method in two dimensions} In two dimensions we solve: \begin{align}\label{eq:2Deq} |\nabla u(x) | & = r(x), \; \;  x \in \Omega=[0,1]^2 \\ \label{eq:2Dbc} u(x) & = 0, \; \; x \in \Gamma \subset [0,1]^2.\end{align} One of the main challenges of setting up the method in two dimensions and higher is setting up the boundary conditions of the subdomains. We approach this by setting up a coarse grid and shifting it vertically and horizontally $M-1$ times each. Let $$\Omega^H=\{(iH,jH) : i,j=0,1,\ldots N\},$$ and $H=1/N.$ Then the horizontally shifted coarse grids are defined by $$\Omega^H+(x_0,0)=\{(iH+x_0,jH) : i=0,1,\ldots N-1, \; j=0,1,\ldots, N\}$$ where $x_0=lh$ for $l=1, \ldots, M-1$ where $h=1/(MN)$. The vertically shifted coarse grids are defined by $$\Omega^H+(0,y_0)=\{(iH,jH+y_0) : i=0,1,\ldots N, \; j=0,1,\ldots, N-1\}$$ where $y_0=mh$ for $m=1, \ldots, M-1$. The shifted grids are demonstrated in \cref{fig:shifted}. 
\begin{figure}\label{fig:shifted}
\begin{tikzpicture}[scale=0.75]
\draw (-2,2)--(2,2)--(2,-2)--(-2,-2)--(-2,2);
\draw (2,0)--(-2,0);
\draw(0,2)--(0,-2);
\draw[dashed] (-1.75,2)--(-1.75,-2);
\draw[dashed] (-1.5,2)--(-1.5,-2);
\draw[dashed] (.25,2)--(.25,-2);
\draw[dashed] (.5,2)--(.5,-2);
\draw[dashed] (2.25,2)--(2.25,-2);
\draw[dashed] (2.5,2)--(2.5,-2);
\draw[dashed] (-1.75,2)--(2.5,2);
\draw[dashed] (-1.75,-2)--(2.5,-2);
\draw[->](-1.5,-1)--(-1,-1);
\draw[->](-1.5,1)--(-1,1);
\draw (0,-3) node{$\Omega^H+(lh,0)$};
\draw (3.5,0) node{+ };
\end{tikzpicture}\;\;
\begin{tikzpicture}[scale=0.75]
\draw (-2,2)--(2,2)--(2,-2)--(-2,-2)--(-2,2);
\draw (2,0)--(-2,0);
\draw(0,2)--(0,-2);
\draw[dashed] (2,-1.75)--(-2,-1.75);
\draw[dashed] (2,-1.5)--(-2,-1.5);
\draw[dashed] (2,.25)--(-2,.25);
\draw[dashed] (2,.5)--(-2,.5);
\draw[dashed] (2,2.25)--(-2,2.25);
\draw[dashed] (2,2.5)--(-2,2.5);
\draw[dashed] (2,-1.75)--(2,2.5);
\draw[dashed] (-2,-1.75)--(-2,2.5);
\draw[->](-1,-1.5)--(-1,-1);
\draw[->](1,-1.5)--(1,-1);
\draw (0,-3) node{$\Omega^H+(0,mh)$};
\draw (3,0) node{= };
\end{tikzpicture}\;\;
\begin{tikzpicture}[scale=0.75]
\draw[step=.25cm,black,  very thin] (-1.75,1.9) grid (-.25,2.1);
\draw[step=.25cm,black,  very thin] (.25,1.9) grid (1.75,2.1);
\draw[step=.25cm,black,  very thin] (-1.75,-1.9) grid (-.25,-2.1);
\draw[step=.25cm,black,  very thin] (.25,-1.9) grid (1.75,-2.1);
\draw[step=.25cm,black,  very thin] (-1.75,-.1) grid (-.25,.1);
\draw[step=.25cm,black,  very thin] (.25,-.1) grid (1.75,.1);
\draw[step=.25cm,black,  very thin] (-1.9,-1.75) grid (-2.1, -.25);
\draw[step=.25cm,black,  very thin] (-1.9,.25) grid (-2.1, 1.75);
\draw[step=.25cm,black,  very thin] (-.1,-1.75) grid (.1, -.25);
\draw[step=.25cm,black,  very thin] (-.1,.25) grid (.1, 1.75);
\draw[step=.25cm,black,  very thin] (1.9,-1.75) grid (2.1, -.25);
\draw[step=.25cm,black,  very thin] (1.9,.25) grid (2.1, 1.75);
\draw (-2,2)--(2,2)--(2,-2)--(-2,-2)--(-2,2);
\draw (2,0)--(-2,0);
\draw(0,2)--(0,-2);
\draw (0,-3) node{ };
\end{tikzpicture}
\caption{Shifted coarse grids in two dimensions}
\end{figure}
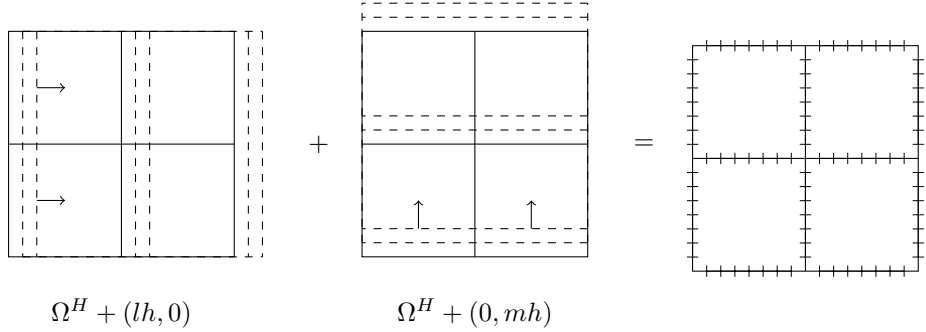
Next we define the fine grids on the subdomains for $i,j=0,1,\ldots, N-1$: $$\Omega^h_{i,j}=\{(lh+iH,mh+jH) : 0\leq l,m\leq M\}.$$ The notation for the two dimensional problem is as follows: \begin{align*} X_{i,j}&=(iH,jH)\in \Omega_H,\\ X_{i_l,j}&=(iH+lh,jH)\in \Omega_H+(lh,0),\\ X_{i,j_m}&=(iH,jH+mh)\in \Omega_H+(0,mh),\\ x_{i_l,j_m}& = (iH+lh,jH+mh) \in \Omega^h_{i,j},\\  U_{i,j}^k& \mbox{ denotes the coarse solution at } X_{i,j} \mbox{ in the } k\mbox{th iteration}, \\ U_{i_l,j}^k& \mbox{ denotes the coarse solution at } X_{i_l,j} \mbox{ in the } k\mbox{th iteration}, \\ U_{i,j_m}^k& \mbox{ denotes the coarse solution at } X_{i,j_m} \mbox{ in the } k\mbox{th iteration}, \\u_{i,j}^k & \mbox{ denotes the fine solution at } X_{i,j} \mbox{ in the } k\mbox{th iteration}, \\ u_{i_l,j_m}^k & \mbox{ denotes the fine solution at } x_{i_l,j_m} \mbox{ in the } k \mbox{th iteration}. \end{align*} Now that we have the grids set up we begin the description of the method. The coarse grid solver is given by the solution to \cref{eq:godonuv}: \begin{equation}C_H(\mbox{nbrs}^H(U_{i,j}))=\begin{cases} \frac{1}{2}\bigg(a+b +\sqrt{2r_{i,j}^2H^2-(a-b)^2}\bigg) & \mbox{ if } |a-b|<r_{i,j}H,\\ \min(a,b)+r_{i,j}H & \mbox{ if } |a-b|\geq r_{i,j}H,\end{cases}\end{equation} where $\mbox{nbrs}^H(U_{i,j})=\{U_{i-1,j},U_{i+1,j},U_{i,j-1},U_{i,j+1}\}$, $a=\min(U_{i-1,j},U_{i+1,j})$, and\\ ${b=\min(U_{i,j-1},U_{i,j+1})}.$  

The steps are the same as in the one dimensional case, except we also have a causal sweep in the initialization step. Step 1 is the initialization of the coarse grids with a causal sweep, step 2 is to update the boundary conditions for the subdomains, and step 3 is to compute the fine solutions on the subdomains in parallel. Step 4 is to perform weighted corrections on the coarse grids  where we allow $\theta$ to vary for each coarse grid node. The weighted update will be: \begin{equation}U_{i,j}^{k+1}=\theta_{i,j}^{k+1} C_H(\mbox{nbrs}^H(U_{i,j}^{k+1}))+u_{i,j}^k-\theta_{i,j}^{k+1} C_H(\mbox{nbrs}^H(U^k_{i,j})).\end{equation}

\subsection*{Step 1: Initialize coarse grids in parallel} Since the shifted coarse grids are independent of each other, the values $\{U^{0}_{i_l,j}\}_{i,j}$ and $\{U^{0}_{i,j_m}\}_{i,j}$ are computed in parallel for each $l$ and $m$.  We solve \cref{eq:2Deq} with boundary conditions \cref{eq:2Dbc} on each of the coarse grids. Keeping track of the flow of characteristics is more complex than in the one dimensional problem. In two dimensions, the set of eight distinct wind direction vectors is $\{(\pm 1,\pm 1), (\pm 1, 0), (0,\pm1)\}$. The wind direction at a coarse grid node is determined by the solution to \cref{eq:godonuv}. We describe how to initialize the grid $\Omega^H$. The initialization on $\Omega^H+(lh,0)$ and $\Omega^H+(0,mh)$ for $l=1,\ldots M-1$ and $m=1,\ldots, M-1$ is the same.
\begin{itemize}
\item Initialize $U$ as described in \cref{sec:FSM}.
\item Sweep the grid as described in \cref{sec:FSM}.  \cref{alg:update2D} explains the update formula as well as how to compute $\mathbf{W}^0_{i,j}$. At each $X_{i,j}$, we input $\mbox{nbrs}^H(U_{i,j}), U_{i,j}$ and $H$, using one sided differences if $X_{i,j}$ is a boundary grid node. 
\end{itemize}
\begin{algorithm}
\caption{Update and wind formula for initialization}
\label{alg:update2D}
\begin{algorithmic}
\STATE{\textbf{Input:} nbrs$^H(U_{i,j}), U_{i,j}, H$}
\STATE{\textbf{Output:} $U_{i,j},\mathbf{W}_{i,j}$}

\COMMENT{Compute wind in $x$ direction.}
\IF{$U_{i-1,j}<U_{i+1,j}$} 
\STATE{$W_x=1$}
\ELSE
\STATE {$W_x=-1$}
\ENDIF

\COMMENT{Compute wind in $y$ direction.}
\IF{$U_{i,j-1}<U_{i,j+1}$}
\STATE{$W_y=1$}
\ELSE
\STATE {$W_y=-1$}
\ENDIF

\COMMENT{Compute in solution to \cref{eq:godonuv} and define $\mathbf{\tilde W}.$}

$\tilde{U}=C_H(\mbox{nbrs}^H(U_{i,j}))$

$a=\min(U_{i-1,j},U_{i+1,j})$

$b=\min(U_{i,j-1},U_{i,j+1})$
\IF{$\tilde U<b$}
\STATE{$\mathbf{\tilde W} = (W_x,0)$}
\ELSIF{$\tilde U <a$}
\STATE{$\mathbf{\tilde W }= (0,W_y)$}
\ELSE
\STATE{$\mathbf{\tilde W }=(W_x,W_y)$}
\ENDIF

\COMMENT{Take minimum.}
\IF{$\tilde U < U_{i,j}$}
\STATE{$U_{i,j}=\tilde U$}
\STATE{$\mathbf{W}_{i,j}=\mathbf{\tilde W }$}
\ENDIF
\end{algorithmic}
\end{algorithm}
The solutions on the coarse grids may have the wrong causality because small scale features in $r_\epsilon$ may be sampled on some shifted coarse grids and not others. To correct this, we implement a causal sweep. We must sweep the coarse grids sequentially in order to capture the right causality. We sweep all the coarse grids in each of the four directions just once. The update is given by inputing $U_{i,j_{m-1}},U_{i,j_{m+1}},U_{i,j_m}, \mathbf{W}_{i,j_m}$ into  \cref{alg:causalsweep2D}, which describes the update for a vertically shifted grid node. The updates for the other coarse grid nodes are defined analogously. Note that since we are sweeping the coarse grids sequentially, $U_{i,j_{m-1}},U_{i,j_{m+1}},$ and $U_{i,j_m}$ belong to three different vertically shifted coarse grids. Denote the solutions after sweeping by $U^0$ and $\mathbf{W}^0$.

\begin{algorithm}
\caption{Causal sweep update formula for vertically shifted coarse grid node}
\label{alg:causalsweep2D}
\begin{algorithmic}
\STATE{\textbf{Input:} $U_{i,j_{m-1}}, U_{i,j_{m+1}},U_{i,j_m},\mathbf{W}_{i,j_m}$}
\STATE{\textbf{Output:} $U_{i,j_m}$}

\IF{$\mathbf{W}_{i,j_m}\cdot (0,-1)>0$ and $U_{i,j_m}< U_{i,j_{m+1}}$}
\STATE{$U_{i,j_m}=U_{i,j_{m+1}}$}
\ENDIF
\IF{$\mathbf{W}_{i,j_m}\cdot (0,1)>0$ and $U_{i,j_m}< U_{i,j_{m-1}}$}
\STATE{$U_{i,j_m}=U_{i,j_{m-1}}$}
\ENDIF
\end{algorithmic}
\end{algorithm}

\subsection*{Step 2: Update boundary conditions for subdomains} 
 Now that we have computed the solutions on all the coarse grids, we can set the boundary conditions for each $\Omega_{i,j}^h$. Intuitively, if a characteristic at a coarse grid point is arriving into the boundary of the subdomain, $\partial \Omega_{i,j}^h$, then we set $u$ at that node to be the value from the coarse grid computations, $U^k$. Otherwise, we set $u$ to be $\infty$ at the coarse grid point. To describe this mathematically for a vertically shifted coarse grid point, define $\mathbf{n}_{w,r_m}$ to be the inward normal vector to the subdomain $\Omega^h_{i,j}$ at $X_{w,r_m} \in \partial \Omega_{i,j}^h$. Then define $$g_{w,r_m}(U_{w,r_m},\mathbf{W}_{w,r_m}):=\begin{cases} U_{w,r_m} &  \mbox{ if } \mathbf{W}_{w,r_m}\cdot \mathbf{n}_{w,r_m} > 0 \\ \infty & \mbox{ otherwise}. \end{cases}$$ When $X_{w_l,r}$ is a horizontally shifted grid point, the definition of $g_{w_l,r}$ is the same as above. For $X_{w,r}$, a non-shifted coarse grid point on $\partial \Omega_{i,j}^h$, the inward normal vector of $\Omega_{i,j}^h$ is not unique since the coarse grid point is a corner of the subdomain. There are two possibilities for the inward normal vector. Denote them by $\mathbf{n}_{w,r}^1$ and $\mathbf{n}_{w,r}^2$, then $$g_{w,r}(U_{w,r}^k,\mathbf{W}_{w,r}^k)=\begin{cases} U_{w,r}^k &  \mbox{ if } \mathbf{W}_{w,r}^k\cdot \mathbf{n}_{w,r}^1 > 0 \mbox{ or }\mathbf{W}_{w,r}^k\cdot \mathbf{n}_{w,r}^2 > 0  \\ \infty & \mbox{ otherwise}. \end{cases}$$

\subsection*{Step 3: Solve for $u^k$ in parallel} In parallel for $i,j=0,\ldots, N-1$, we solve \begin{align} |\nabla u(x) | &=r(x), \; \; x\in [iH,(i+1)H]\times[jH,(j+1)H] \\ u&=g,\; \; \mbox{on } \partial ( [iH,(i+1)H]\times[jH,(j+1)H] )\end{align} via FSM on the grid $\Omega_{i,j}^h$ and $g$ is defined in step 2.
\begin{itemize} 
\item Initialize $u$ as described in \cref{sec:FSM}
\item Sweep the grid. To the update the solution at each $x_{i_l,j_m}$, input $\mbox{nbrs}^h(u_{i_l,j_m})$, $u_{i_l,j_m}$, and $h$ into \cref{alg:fineupdate2D}. Use one sided differences if $x_{i_l,j_m}$ is a boundary grid node. Here, $\mbox{nbrs}^h(u_{i_l,j_m})=\{u_{i_{l-1},j_m},u_{i_{l+1},j_m},u_{i_l,j_{m-1}},u_{i_l,j_{m+1}}\}$. \end{itemize}
Denote the solutions after sweeping by $u^k_{i_l,j_m}$ and $\mathbf{w}^k_{i_l,j_m}$ for $l,m=0,\ldots M$.

\begin{algorithm}
\caption{Update and wind formula for fine grid compuations}
\label{alg:fineupdate2D}
\begin{algorithmic}
\STATE{\textbf{Input:} $\mbox{nbrs}^h(u_{i_l,j_m}), u_{i_l,j_m}, h$}
\STATE{\textbf{Output:} $u_{i_l,j_m},\mathbf{w}_{i_l,j_m}$}

\COMMENT{Compute wind in $x$ direction.}
\IF{$u_{i_{l-1},j}<u_{i_{l+1},j}$} 
\STATE{$a=u_{i_{l-1},j}$}
\STATE{$w_x=-1$}
\ELSE
\STATE{$a=u_{i_{l+1},j}$}
\STATE {$w_x=1$}
\ENDIF

\COMMENT{Compute wind in $y$ direction.}
\IF{$u_{i,j_{m-1}}<u_{i,j_{m+1}}$}
\STATE{$b=u_{i,j_{m-1}}$}
\STATE{$w_y=-1$}
\ELSE
\STATE{$b=u_{i,j_{m+1}}$}
\STATE {$w_y=1$}
\ENDIF

\COMMENT{Solve \cref{eq:godonuv}.}

\IF{$|a-b|<r_{i,j}h$}
\STATE{$\tilde u= \frac{1}{2}\bigg(a+b +\sqrt{2r_{i,j}^2h^2-(a-b)^2}\bigg)$}
\STATE{$\mathbf{\tilde w }=(w_x,w_y)$}
\ELSE
\STATE{$ \tilde u=\min(a,b)+r_{i,j}h$}
\STATE{ \IF{$a<b$} \STATE{$\mathbf{\tilde w} = (w_x,0)$} \ELSE \STATE{$\mathbf{\tilde w }= (0,w_y)$} \ENDIF}
\ENDIF

\COMMENT{Take minimum.}
\IF{$\tilde u < u_{i_l,j_m}$}
\STATE{$u_{i_l,j_m}=\tilde u$}
\STATE{$\mathbf{w}_{i_l,j_m}=\mathbf{\tilde w }$}
\ENDIF
\end{algorithmic}
\end{algorithm}

After the computations on each subdomain, we will have two or four values for each coarse grid node, depending on whether the point is in a shifted or non-shifted coarse grid. Intuitively, we define the value $u_{i,j}^k$ by the following:
\begin{itemize} 
\item If the coarse wind and the fine wind flow into the same subdomain $\Omega_{s,t}^h$ from $\Omega_{s',t'}^h$, then we set the value $u_{i,j}^k$ to be the fine grid solution from the subdomain $\Omega_{s',t'}^h$. 
\item Otherwise we set $u_{i,j}^k$ to be the minimum of the fine grid solutions at the coarse grid point.
\end{itemize}
 A vertically shifted coarse grid node, $X_{i,j_m}$, is on the boundary of the two subdomains, $\Omega_{i-1,j'}^h$ and $\Omega_{i,j'}^h$. Denote the two possibilities of an inward normal vector by ${\mathbf{n}^1=(-1, 0)}$ and $\mathbf{n}^2=(1,0)$.  \cref{alg:fineupdate} explains how to compute $u_{i,j_m}^k$ at a vertically shifted coarse grid node, $X_{i,j_m}$. The computations  at a horizontally shifted and non shifted coarse grid point are similar. 
\begin{algorithm}
\caption{Update formula for $u_{i,j_m}^k$ and $\mathbf{w}_{i,j_m}^k$ for a vertically shifted coarse grid node}
\label{alg:fineupdate}
\begin{algorithmic}
\STATE{\textbf{Input:} $u_{{i-1}_M,j_m}^k, u_{i_0,j_m}^k, \mathbf{w}_{{i-1}_M,j_m}^k, \mathbf{w}_{i_0,j_m}^k, \mathbf{W}_{i,j_m}^k$}
\STATE{\textbf{output:} $u_{i,j_m}^k, \mathbf{w}_{i,j_m}^k$}
\IF{$\mathbf{w}_{{i-1}_M,j_m}^k \cdot \mathbf{n}^1 \geq 0, \mathbf{w}_{i_0,j_m}^k \cdot  \mathbf{n}^1 \geq 0 ,$ and $\mathbf{W}_{i,j_m}^k \cdot \mathbf{n}^1\geq 0$}
\STATE{$u_{i,j_m}^k=u_{{i}_0,j_m}^k$}
\STATE{$\mathbf{w}_{i,j_m}^k=\mathbf{w}_{{i}_0,j_m}^k$}
\ELSIF{$\mathbf{w}_{{i-1}_M,j_m}^k \cdot \mathbf{n}^2 \geq 0, \mathbf{w}_{i_0,j_m}^k \cdot  \mathbf{n}^2 \geq 0 ,$ and $\mathbf{W}_{i,j_m}^k \cdot \mathbf{n}^2\geq 0$}
\STATE{$u_{i,j_m}^k=u_{{i-1}_M,j_m}^k$}
\STATE{$\mathbf{w}_{i,j_m}^k=\mathbf{w}_{{i-1}_M,j_m}^k$}
\ELSE 
\STATE{\IF{$u^k_{{i-1}_M,j_m}\leq u_{i_0,j_m}^k$}
	\STATE{$u_{i,j_m}^k=u^k_{{i-1}_M,j_m}$}
	\STATE{$\mathbf{w}_{i,j_m}^k=\mathbf{w}_{{i-1}_M,j_m}^k$}
	\ELSE
	\STATE{$u_{i,j_m}^k=u^k_{i_0,j_m}$}
	\STATE{$\mathbf{w}_{i,j_m}^k=\mathbf{w}_{i_0,j_m}^k$}
	\ENDIF
}
\ENDIF
\end{algorithmic}
\end{algorithm}
\subsection*{Step 4: Coarse grid updates} Now we compute the coarse grid updates, $U^{k+1}$. Again since the shifted coarse grids are independent of each other, the values $\{U^{k+1}_{i_l,j}\}_{i,j}$ and $\{U^{k+1}_{i,j_m}\}_{i,j}$ can be computed in parallel for each $l$ and $m$.
\begin{itemize}
\item Initialize $U$ as described in  \cref{sec:FSM}. 
\item Sweep the grid and the update formula at a vertically shifted coarse grid node is given by \cref{alg:2Dweightedupdate}. The computations at a horizontally shifted and non shifted coarse grid point are similar. Let $\mathbf{n}^1$ and $\mathbf{n}^2$ be the inward normal vectors as defined in step 3. We input $\mbox{nbrs}^H(U_{i,j_m}), \mbox{nbrs}^H(U^k_{i,j_m}),\mathbf{W}_{i,j_m}^k, \mathbf{w}_{i,j_m}^k$ and $H$ into  \cref{alg:2Dweightedupdate}.
\end{itemize}

\begin{algorithm}
\caption{Update formula for weighted corrections for a vertically shifted coarse grid node}
\label{alg:2Dweightedupdate}
\begin{algorithmic}
\STATE{\textbf{Input:} $\mbox{nbrs}^H(U_{i,j_m}), \mbox{nbrs}^H(U^k_{i,j_m})$, $W_{i,j_m}^k, w_{i,j_m}^k, H$}
\STATE{\textbf{Output:} $U_{i,j_m},W_{i,j_m}$}

\STATE{Compute $\tilde U$ and $\tilde W$ as in \cref{alg:update2D}}

\IF{$\mathbf{w}_{i,j_m}^k \cdot \mathbf{n}^1 \geq 0, \mathbf{W}_{i,j_m}^k \cdot  \mathbf{n}^1 \geq 0 ,$ and $\mathbf{\tilde W} \cdot \mathbf{n}^1\geq 0$}
\STATE{$U_{i,j_m}=\theta_{i,j_m}^{k+1} \tilde U +u_{i,j_m}^k-\theta_{i,j_m}^{k+1}  C_H(\mbox{nbrs}^H(U^k_{i,j_m}))$}
\STATE{$\mathbf{W}_{i,j_m}=\mathbf{w}_{i,j_m}^k$}
\ELSIF{$\mathbf{w}_{i,j_m}^k \cdot \mathbf{n}^2 \geq 0, \mathbf{W}_{i,j_m}^k \cdot  \mathbf{n}^2 \geq 0 ,$ and $\mathbf{\tilde W} \cdot \mathbf{n}^2\geq 0$}
\STATE{$U_{i,j_m}=\theta_{i,j_m}^{k+1}  \tilde U +u_{i,j_m}^k-\theta_{i,j_m}^{k+1}  C_H(\mbox{nbrs}^H(U^k_{i,j_m}))$}
\STATE{$\mathbf{W}_{i,j_m}=\mathbf{w}_{i,j_m}^k$}
\ELSE
\STATE{$U_{i,j_m}=u_{i,j_m}^k$}
\STATE{$\mathbf{W}_{i,j_m}=\mathbf{w}_{i,j_m}^k$}
\ENDIF

\end{algorithmic}
\end{algorithm}

Again we must implement a sequential causal sweep to make sure the coarse grids respect the causality of the solution. Sweep the coarse grids sequentially in each of the four directions once. The update formula is given by inputing $U_{i,j_{m-1}},U_{i,j_{m+1}},U_{i,j_m}, \\\mathbf{W}_{i,j_m}$ into  \cref{alg:causalsweep2D} for a vertically shifted coarse grid point. The updates for other coarse grid nodes are defined similarly. Denote the solutions after sweeping by $U^{k+1}$ and $\mathbf{W}^{k+1}$. Repeat steps 2-4 until convergence.
\begin{figure}[h]
\centering
\subfloat[$k=0$]{\includegraphics{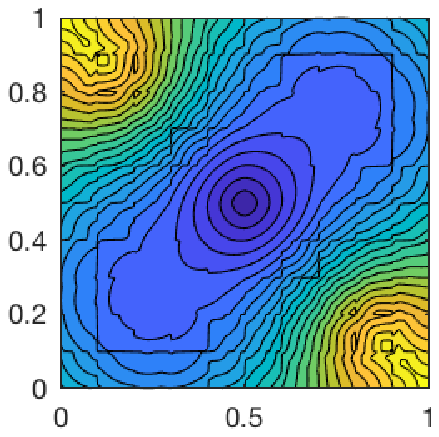}\label{fig:a12}}
\subfloat[$k=2$]{\includegraphics{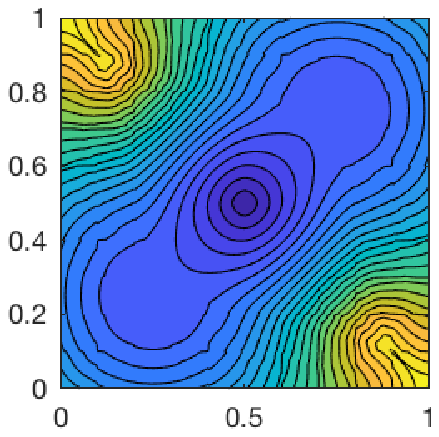}\label{fig:b12}}\\
\subfloat[$k=4$]{\includegraphics{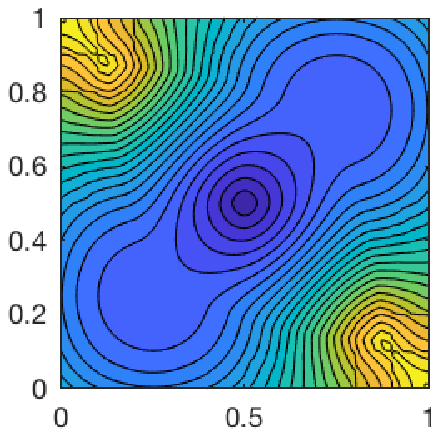}\label{fig:c12}}
\subfloat[$k=6$]{\includegraphics{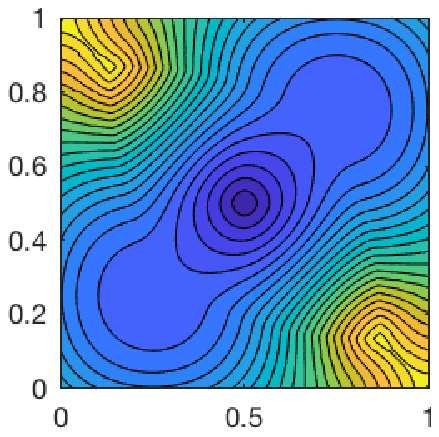}\label{fig:d12}}
\caption{Contours for fine grid solutions patched together for the slowness function ${r^1_\epsilon(x,y)=1+.99\sin(2\pi x)\sin(2\pi y)}$ where in  \protect\subref{fig:a12}  $k=0$, \protect\subref{fig:b12} $k=2$,\protect\subref{fig:c12} $k=4$ and \protect\subref{fig:d12} $k=6$.}
\label{fig:solutioncontoursk}
\end{figure}

The method is demonstrated in \cref{fig:solutioncontoursk} which shows the contours for the fine grid solution patched together for ${r^1_\epsilon=1+.99\sin(2\pi x)\sin(2\pi y)}$ for $k=0,2,4$ and $6$. We see the solution contours begin to smooth out after a few iterations.

\section{Analysis of the new method}
\label{sec:analysis}
\begin{figure}[h]
\centering
\includegraphics{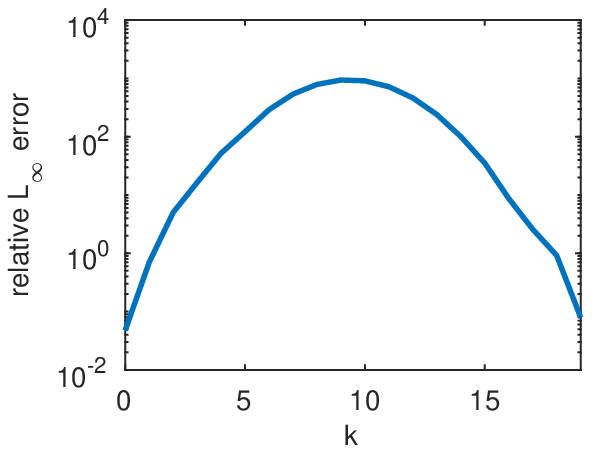}
\caption{$||U^k-u^f||_{L_\infty}$ for $\theta=1, H=1/20, h=1/1000$.}
\label{fig:theta1}
\end{figure}
We choose the following model problem to study the choice of weight $\theta$. Let $\Omega =[0,1]\times[0,H ]$ Then we numerically solve via our method \begin{align} |\nabla u(x,y)| & =1, \; \; (x,y)\in \Omega \backslash \Gamma \\  u(x,y)& =\sqrt{x^2+y^2}, \;\; (x,y) \in \Gamma  \end{align} where $\Gamma= \{(x,0): 0\leq x \leq 1\}\cup \{(0,y): 0\leq y \leq H\}$. The coarse grids can be defined in one set by $$\Omega^H:= \{(iH, jh):i=0,1\ldots,N \mbox{ and } j=0,1,\ldots,M\}$$ with $X_{i,j}=(iH,jh)$. The overall fine grid is given by $$\Omega^h=\{ (lh,mh): l=0,1,\ldots, NM \mbox{ and } m=0,1,\ldots M \}.$$ The advantage of this problem is that the characteristics can be captured in one sweep of FSM, i.e., an upward right sweep. This fact means we can use a weighted correction update for every coarse grid node. Let $u^f$ be the overall fine solution on $\Omega^h$. Suppose we allow $\theta$ to vary for each coarse grid node and iteration and denote it by $\theta_{i,j}^k$. Then for $i=1,\ldots,  N$ and $j=1,\ldots, M$, we have the following coarse grid solver:  $$ C_H(U_{i-1,j},U_{i,j-M})=\begin{cases} \frac{U_{i-1,j}+U_{i,j-M}+\sqrt{2H^2-(U_{i-1,j}-U_{i,j-M})^2}}{2}, & j= M \\ U_{i-1,j}+H, & \mbox{ otherwise }\end{cases},$$ where if $j=1,\ldots, M-1$ we ignore the second argument of the coarse grid solver. Let $U_{i,j}^0= C_H(U^0_{i-1,j},U^0_{i,j-M})$. The weighted update for this problem for $j=1,\ldots M:$ \begin{equation}\label{eq:modelprob} U_{i,j}^{k+1}=\theta_{i,j}^{k+1} \bigg [C_H(U_{i-1,j}^{k+1},U_{i,j-M}^{k+1}) -C_H(U_{i-1,j}^{k},U_{i,j-M}^{k})\bigg ] + u_{i,j}^{k}\end{equation} with initial conditions \begin{equation}\label{eq:modelBC1} U_{0,j}^{k+1}=u_{0,j}^f  \mbox{ for } j=1, \ldots, M \mbox{ and } k=0,1,2\ldots \end{equation} and \begin{equation}\label{eq:modelBC2} U_{i,0}^{k+1}=u_{i,0}^f \mbox{ for }  i=1, \ldots, N \mbox{ and } k=0,1,2.\ldots\end{equation}

 \cref{fig:theta1} shows the $L_\infty$ error plot for $\theta= 1$ which is analogous to the standard parareal method. Note the error is large from the first iteration and increases for later iterations. The error peaks around $k=10$. This is because as $k$ increases the solutions in each successive subdomain converge to the exact solution  which then allows the maximum error to begin to decrease.  If we choose a small value for $\theta$, the solutions converge as seen in \cref{fig:vartheta}. However, the convergence may be slow. Next, we study how to choose $\theta_{i,j}^k$ for the coarse grid updates.

\subsection{Analysis of $\theta$ on model problem}

First we prove a theorem that gives an exactness property for the method on this model problem. Let $u^f$ be the overall fine solution on $\Omega^h$.
\begin{theorem}\label{thm:exactness}
Let $U_{i,j}^k$ be given by \cref{eq:modelprob}. Then for each $j=1,\ldots M,$ $$U_{i,j}^k=u_{i,j}^f \mbox{ for } k \geq i.$$ \end{theorem}

\begin{proof}
First note $U_{0,j}^1=U_{0,j}^0$ and $U_{1,0}^1=U_{1,0}^0=u_{1,0}^f$.  Now,  \begin{align*} U_{1,j}^1 & =\theta_{1,j}^1\bigg [C_H(U_{0,j}^1,U_{1,j-M}^1)-C_H(U_{0,j}^0,U_{1,j-M}^0)\bigg ]+u_{1,j}^0 \\ & = u_{1,j}^0 \\ &=u_{1,j}^f.\end{align*} The second equality comes from the fact that $u_{1,j}^0$ was computed using only the boundary values.

Now assume  \begin{equation}\label{eq:exactHYP} U_{i,j}^k=u_{i,j}^f \mbox{ for } k \geq i. \end{equation} Let $k \geq i+1$. We have $$U_{i+1,j}^k=\theta_{i+1,j}^k\bigg [C_H(U_{i,j}^k,U_{i+1,j-M}^k)-C_H(U_{i,j}^{k-1},U_{i+1,j-M}^{k-1})\bigg ] + u_{i+1,j}^{k-1}.$$ Now $k\geq i+1$ and  \cref{eq:exactHYP} imply $U_{i,j}^k=U_{i,j}^{k-1}=u_{i,j}^f.$ Also, \cref{eq:modelBC2} imply $$U_{i+1,0}^k=U_{i+1,0}^{k-1}=u_{i+1,0}^f.$$ Therefore, $$U_{i+1,j}^k=u_{i+1,j}^{k-1}=u_{i+1,j}^f,$$ where second equality comes from  the fact that $u_{i+1,j}^{k-1}$ is computed using the values $U_{i,j}^{k-1}$ and \cref{eq:exactHYP} implies $U_{i,j}^{k-1}=u_{i,j}^f$ for $j=0,\ldots,M-1$. Thus, we have our desired result.
\end{proof}

Now that the exactness property for the method is proven, we have the following theorem that proves existence of $\theta_{i,j}^k$ for each $k$  such that the sequence of solutions is monotonically decreasing for the model problem. The following fact is used in the proof of the theorem: If $a\leq b \leq c$, then $C_H(a,c)\leq C_H(b,c)$.
\begin{theorem}\label{thm:choosetheta} For $j=1,\ldots, M$ and  $i=1,\ldots,N$, there exists $\theta_{i,j}^k$ such that $U_{i,j}^{k} < U_{i,j}^{k-1}$ and $U_{i,j}^k> u^f_{i,j}$ for all  $i> k$. 
\end{theorem}

\begin{proof}
First we note $U_{i,j}^0 > u_{i,j}^f$ and  $u_{i,j}^0 \geq u_{i,j}^f$ for $i=1,\ldots N$ and $j=1,\ldots, M $. We will proceed by induction on $k$. Let $k=1$. We will show the theorem holds for all $i > 1$. Either $u_{2,j}^0 > U_{2,j}^0$ or $u_{2,j}^0 \leq U_{2,j}^0$. If $u_{2,j}^0 \leq U_{2,j}^0$, choose $\theta_{2,j}^1>0$. Then \begin{align*} U_{2,j}^1&=\theta_{2,j}^1 \bigg[C_H(U_{1,j}^1,U_{2,j-M}^1)-C_H(U_{1,j}^0,U_{2,j-M}^0)\bigg ]+u_{2,j}^0\\& < u_{2,j}^0 \\ & \leq U_{2,j}^0\end{align*}  where the first inequality comes from the fact that $U_{1,j}^1=u_{1,j}^f < U_{1,j}^0$. If $u_{2,j}^0>U_{2,j}^0$, then define $$\overline{m}_{2,j}^1=\frac{U_{2,j}^0-u_{2,j}^0}{C_H(U_{1,j}^1,U_{2,j-M}^1)-C_H(U_{1,j}^0,U_{2,j-M}^0)}.$$ Now $\overline{m}_{2,j}^1>0$. Thus, if $\theta_{2,j}^1 > \overline{m}_{2,j}^1$, then $U_{2,j}^1<U_{2,j}^0$. Now let $$\overline{M}_{2,j}^{1}=\frac{u_{2,j}^f-u_{2,j}^0}{C_H(U_{1,j}^1,U_{2,j-M}^1)-C_H(U_{1,j}^0,U_{2,j-M}^0)}.$$ Note $\overline{M}_{2,j}^{1}>0$. If we choose $\theta_{2,j}^1<\overline{M}_{2,j}^1$, then $U_{2,j}^1>u_{2,f}^f$.

Now assume $U_{i,j}^1>U_{i,j}^0$ and $U_{i,j}^1>u_{i,j}^f$ for all $i>1$. If $u_{i+1,j}^0 \leq U_{i+1j}^0$, choose $\theta_{i+1,j}^1>0$. Then \begin{align*} U_{i+1,j}^1& =\theta_{i+1,j}^1\bigg[C_H(U_{i,j}^1,U_{i+1,j-M}^1)-C_H(U_{i,j}^0,U_{i+1,j-M}^0)\bigg ] +u_{i+1,j}^0 \\ & < u_{i+1,j}^0 \\ & \leq  U_{i+1,j}^0, \end{align*} where the first inequality comes from the induction assumption. If $u_{i+1,j}^0 > U_{i+1j}^0$, then let $$\overline{m}_{i+1,j}^1=\frac{U_{i+1,j}^0-u_{i+1,j}^0}{C_H(U_{i,j}^1,U_{i+1,j-M}^1)-C_H(U_{i,j}^0,U_{i+1,j-M}^0)}$$ Now $\overline{m}_{i+1,j}^1>0$. Thus, if $\theta_{i+1,j}^1 > \overline{m}_{i+1,j}^1$, then $U_{2,j}^1<U_{2,j}^0$. Next let $$\overline{M}_{i+1,j}^1=\frac{u_{i+1,j}^f-u_{i+1,j}^0}{C_H(U_{i,j}^1,U_{i+1,j-M}^1)-C_H(U_{i,j}^0,U_{i+1,j-M}^0)}.$$ Note $\overline{M}_{i+1,j}^{1}>0$. If we choose $\theta_{i+1,j}^1 < \overline{M}_{i+1,j}^1$, then $U_{i+1,j}^1>u_{i+1,j}^f.$  So if \\${0<\theta_{i+1,j}^1<\overline{M}_{i+1,j}^1}$, the theorem holds for $k=1$.

Assume the theorem holds for $k$, i.e., $U_{i,j}^k< U_{i,j}^{k-1}$ and $U_{i,j}^k>u_{i,j}^f$ for $i>k$. We want to show  it holds  $i>k+1$. Let $i=k+2$. The induction hypothesis implies $U_{k+1,j}^k>u_{k+1,j}^f$ and \cref{thm:exactness}  implies $U_{k+1,j}^{k+1}=u_{k+1,j}^f$. Thus, $$C_H(U_{k+1,j}^{k+1},U_{k+2,j-M}^{k+1})-C_H(U_{k+1,j}^k,U_{k+2,j-M}^k) <0.$$ If $u_{k+2,j}^k\leq U_{k+2,j}^k$, choose $\theta_{k+2,j}^{k+1}>0$. Then  \begin{align*} U_{k+2,j}^{k+1}& =\theta_{k+2,j}^{k+1}\bigg[C_H(U_{k+1,j}^{k+1},U_{k+2,j-M}^{k+1})-C_H(U_{k+1,j}^k,U_{k+2,j-M}^k)\bigg ] +u_{k+1,j}^k \\ & < u_{k+2,j}^k \\ & \leq U_{k+2,j}^k\end{align*} If $u_{k+2,j}^k>U_{k+2,j}^k$, let $$\overline{m}_{k+2,j}^{k+1}=\frac{U_{k+2,j}^k-u_{k+2,j}^k}{C_H(U_{k+1,j}^{k+1},U_{k+2,j-M}^{k+1})-C_H(U_{k+1,j}^k,U_{k+2,j-M}^k)}.$$ Note $\overline{m}_{k+2,j}^{k+1}>0$. If $\theta_{k+2,j}^{k+1}>\overline{m}_{k+2,j}^{k+1}$, then $U_{k+2,j}^{k+1}<U_{k+2,j}^k.$ 

Next we need ${U_{k+2,j}^{k+1}>u_{k+2,j}}^f$. Let $$\overline{M}_{k+2,j}^{k+1}=\frac{u_{k+2,j}^f-u_{k+2,j}^k}{C_H(U_{k+1,j}^{k+1},U_{k+2,j-M}^{k+1})-C_H(U_{k+1,j}^k,U_{k+2,j-M}^k)}.$$ Note $\overline{M}_{k+2,j}^{k+1}$. Then if $\theta_{k+2,j}^{k+1}<\overline{M}_{k+2,j}^{k+1}$, $U_{k+2,j}^{k+1}>u_{k+2,j}^f.$ So if $u_{k+2,j}^k\leq U_{k+2,j}^k,$ choose $0< \theta_{k+2,j}^{k+1}<\overline{M}_{k+2,j}^{k+1}$. If $u_{k+2,j}^k>U_{k+2,j}^k,$ choose $\overline{m}_{k+2,j}^{k+1}<\theta_{k+2,j}^{k+1}<\overline{M}_{k+2,j}^{k+1}$. Therefore, the theorem holds.
\end{proof}

The proof of  \cref{thm:choosetheta} provides insight on the stability of the method and the optimal choice for the weights $\theta_{i,j}^k$. Let $$\widetilde{m}_{i,j}^{k}=\begin{cases} 0 & \mbox{ if } \overline{m}_{i,j}^k \leq 0 \\\overline{m}_{i,j}^{k} & \mbox{ otherwise }\end{cases}.$$ If $\widetilde{m}_{i,j}^{k}<\theta_{i,j}^k<\overline{M}_{i,j}^{k}$, we have a monotonically convergent sequence of solutions. The closer we choose $\theta_{i,j}^k$ to $\overline{M}_{i,j}^{k}$ the more accurate $U_{i,j}^k$ is.

If we analyze the values for $\overline{M}_{i,j}^{k}$ we see in early iterations that $\overline{M}_{i,j}^{k}$ can be very small. For example, if $k=1, h=1/20, h=1/1000$, $\min_{X_{i,j} \in \Omega^H}(\overline{M}_{i,j}^{k})=5.6\times10^{-3}$. This is a reason why we cannot use the standard parareal method where $\theta=1$. In practice, we do not know $\overline{M}_{i,j}^{k}$ a priori since it relies on knowing $u_{i,j}^f$.  Therefore, we estimate $\overline{M}_{i,j}^{k}$ in order to choose $\theta_{i,j}^k$ and create a sequence $U_{i,j}^k$ that converges very quickly to $u_{i,j}^f$. Next, we explain how we estimate $\overline{M}_{i,j}^{k}$ in practice. 

\subsection{Estimating $\overline{M}_{i,j}^{k}$}
Recall $$\overline{M}_{i,j}^{k}=\frac{u_{i,j}^f-u_{i,j}^{k-1}}{C_H(U_{i-1,j}^{k},U_{i,j-M}^{k})-C_H(U_{i-1,j}^{k-1},U_{i,j-M}^{k-1})}$$ and we would like $\tilde{m}_{i,j}^k \leq \theta_{i,j}^{k}<\overline{M}_{i,j}^{k}$. 
Since we do not know $u_{i,j}^f$ a priori, we estimate $\overline{M}_{i,j}^{k}$ by the following \begin{equation} \label{eq:firstthetaest}\overline{\theta}_{i,j}^k=\frac{u_{i,j}^{k-1}-u_{i,j}^{k-2}}{C_H(U_{i-1,j}^{k-1},U_{i,j-M}^{k-1})-C_H(U_{i-1,j}^{k-2},U_{i,j-M}^{k-2})}.\end{equation} However, upon implementation this estimation produces very unstable solutions. The values $\overline{\theta}_{i,j}^k$ become extremely large and creates sequences of solutions where $U_{i,j}^k \ll u_{i,j}^f$ or $U_{i,j}^k \gg U_{i,j}^{k-1}$. This occurs when the denominator of \cref{eq:firstthetaest} is much smaller than the numerator. We overcome this issue by using a weighted sum in the denominator, i.e., \begin{equation}\label{eq:secthetaest}\overline{\theta}_{i,j}^k= \frac{u_{i,j}^{k-1}-u_{i,j}^{k-2}}{\bigg [ \sum_{s=0}^2 \omega_s C_H(U_{i-1,j}^{k-s},U_{i,j-M}^{k-s})-C_H(U_{i-1,j}^{k-1-s},U_{i,j-M}^{k-1-s})\bigg ]/(\omega_0+\omega_1+\omega_2) }.\end{equation} To further ensure that the estimated value, $\overline{\theta}_{i,j}^k$ does not become too large we dampen the values if they are beyond a threshold and apply a smooth approximation function. Let \begin{equation}\sigma(\overline{\theta}_{i,j}^k)=\frac{1}{1+e^{\overline{(\theta}_{i,j}^k-x_0)/\gamma}}. \end{equation}  Then \begin{equation}\label{eq:thetaest} \overline{\theta}_{i,j}^{k,used} =\left [\sigma(\overline{\theta}_{i,j}^k)   \overline{\theta}_{i,j}^k+(1-\sigma(\overline{\theta}_{i,j}^k) )\delta\overline{\theta}_{i,j}^k\right]^+,\end{equation} where $x_0, \gamma,$ and $\delta$ are parameters chosen experimentally. 

\begin{figure}[h]
\centering
\subfloat[$H=1/10$]{\includegraphics{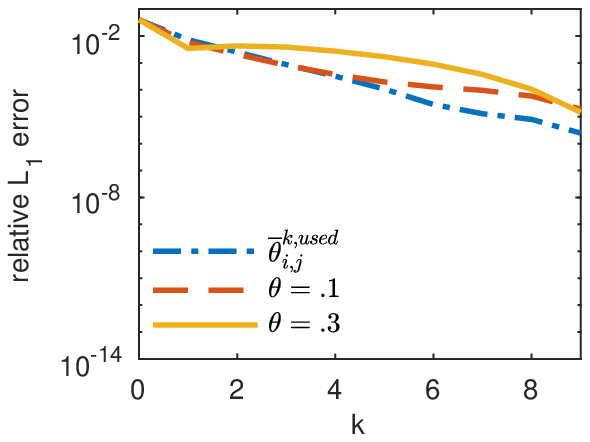}\label{fig:a1}}
\subfloat[$H=1/20$]{\includegraphics{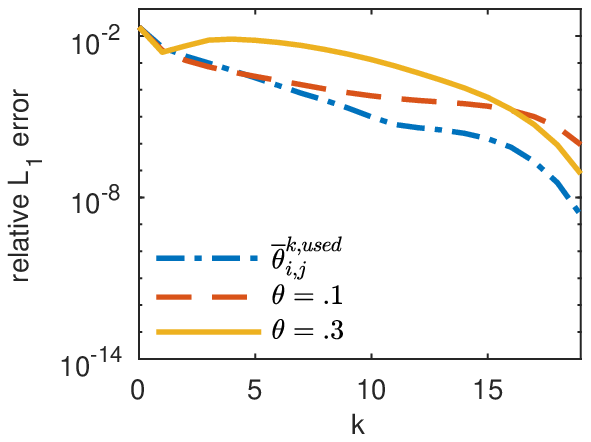}\label{fig:b1}}\\
\subfloat[$H=1/50$]{\includegraphics{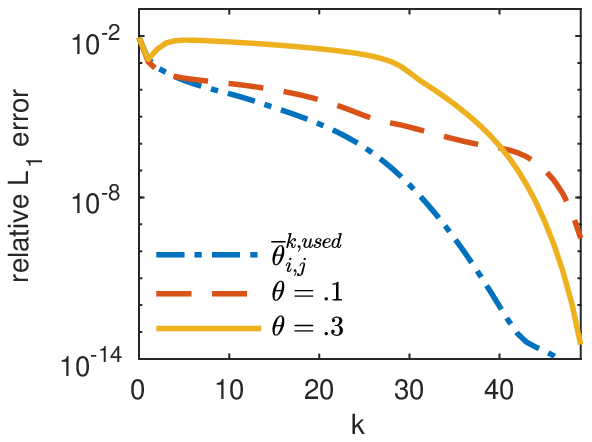}\label{fig:c1}}
\subfloat[]{\includegraphics{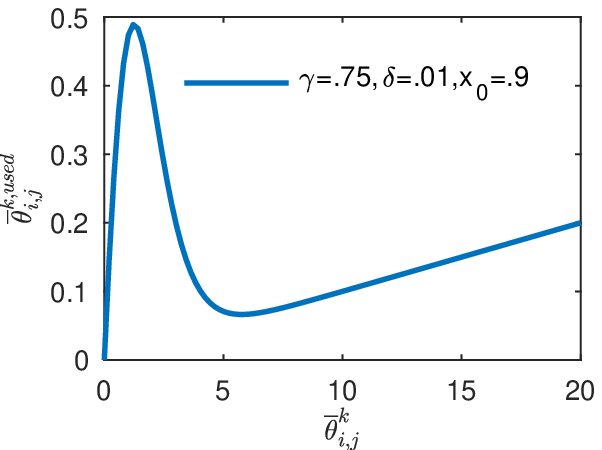}\label{fig:d1}}
\caption{Error plots of $||U^k-u^f||_{L_1}$ for specified values of $H$. In all three examples $h=1/1000$.  For  \protect\subref{fig:a1}, \protect\subref{fig:b1}, and  \protect\subref{fig:c1}, ${||u^f-u^{exact}||_{L_1}=3.79\times10^{-4}}$. \protect\subref{fig:d1} shows the  parameters $\gamma,\delta,$ and $x_0$ used to estimate $\theta$. For \protect\subref{fig:a1}- \protect\subref{fig:c1}, $\gamma=.75,\delta=.01,x_0=.9$. }
\label{fig:vartheta}
\end{figure}

\cref{fig:d1} shows the plot of $\overline{\theta}_{i,j}^k$ versus $\overline{\theta}_{i,j}^{k,used}$, and \cref{fig:a1,fig:b1,fig:c1} show the error plots for various  values of $H$. In all three examples $h=1/1000$. We see the advantage of using  $\overline{\theta}_{i,j}^{k,used}$ over a fixed value of $\theta$.

\subsection{Complexity and speed up}

Let $N=1/H$ and $M=1/(Nh)$. Define $A(N,d):=C(2^d(N+1)^d)$ be the number of flops for FSM where $C$ depends on the characteristics of the given Eikonal equation. Then the number of flops for the computations on all of the coarse grids and the causal sweep is $(dM)A(N,d)+2^dM(N+1)^2$ and the number of flops for all the subdomains combined is $(N^d)A(M,d)$. If we solve the Eikonal equation on $\Omega_h$, then the number of flops is given by $A(NM,d)$. Theoretically suppose we have enough processors to compute the solution on each coarse grid and each subdomain in parallel. Then after $k$ iterations the number of flops will be $k[A(N,d)+A(M,d)+2^dM(N+1)^2]$. For our method to achieve speed up via parallelization, we need $$k \ll \frac{A(NM,d)}{A(N,d) +A(M,d)+2^dM(N+1)^2}.$$ For example, suppose $N=20, M=100,$ and $d=2$ and we perform 10 sweeping iterations on each coarse grid as well as on each subdomain, then we need $k \ll 266$ in order to achieve speed up.

\section{Numerical results}
\label{sec:results} 
\footnote{Matlab/C++ code used to produce all numerical results can be found at https://github.com/lindsmart/MartinTsaiEikonal}Next we present some numerical results computed by the method. Every example is computed on $\Omega=[0,1]^2$ and $\Gamma$ is a set of source points chosen in each example. The focus of our examples is demonstrating the reduction in error in a few iterations and the ability to handle some stereotypes of $r_\epsilon$. We report the $L_1$ relative error in each example, i.e., $||U^k-u^f||_{L_1}$ where $u^f$ is the overall fine solution. In every example, $\overline{\theta}_{i,j}^{k,used}$ is chosen so the solution is stable and converges to the overall fine solution. 

\subsection{Smooth slowness functions}\label{sec:smoothslow}

 We test the method on two smooth oscillatory continuous slowness functions. \cref{fig:a2,fig:b2} show the contour plots of the overall fine solution where $${r^1_\epsilon(x,y)=1+.99\sin(2\pi x)\sin(2 \pi y)}$$ and $$r^2_\epsilon(x,y)=1+.5\sin(20\pi x)\sin(20 \pi y).$$  \cref{fig:c2} shows the error plots for $H=1/10$ and $h=1/500$. The method is able to handle small and large changes in direction of the characteristics. We see that the performance is worse in earlier iterations for $r_\epsilon^1$. This is because the solutions in the upper left and bottom right corner depend on more subdomains than in the $r_\epsilon^2$ case. 
 \begin{figure}[h]
\centering
\subfloat[$r_\epsilon^1$]{\includegraphics{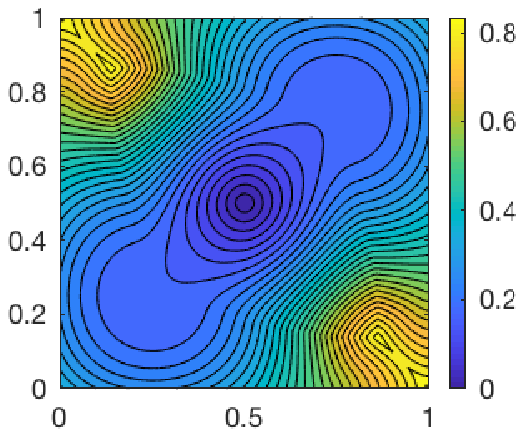}\label{fig:a2}}
\subfloat[$r_\epsilon^2$]{\includegraphics{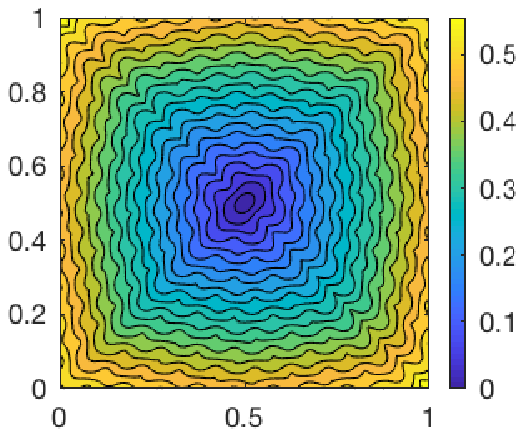}\label{fig:b2}} \;\;\;
\caption{ \protect\subref{fig:a2} Solution contour for $r^1_\epsilon(x,y)=1+.99\sin(2\pi x)\sin(2 \pi y)$. \protect\subref{fig:b2}  Solution contour for $r^2_\epsilon(x,y)=1+.5\sin(20\pi x)\sin(20 \pi y)$.}
\label{fig:smoothosc}
\end{figure}

\begin{figure}[h]
\centering
\includegraphics{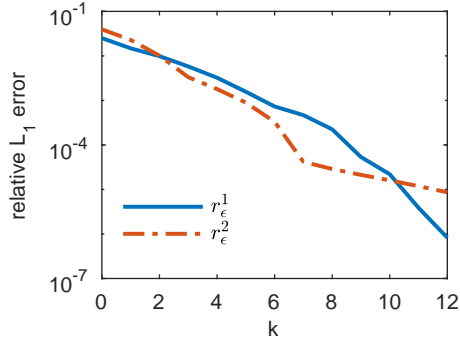}
\caption{ Relative $L_1$ error plots for  $r^1_\epsilon$ and $r^2_\epsilon$ for $H=1/10$ and $h=1/500$.}
\label{fig:c2}
\end{figure}

\subsection{Mazes and obstacles}
We show the method's performance on examples that model optimal paths through a maze. Here, we define $r_\epsilon(x,y)=1000$ inside the barriers so that all optimal paths choose to avoid them. We also test the method on the case where an obstacle may be a ``fast obstacle", i.e., $r_\epsilon(x,y)=0.01$ inside and optimal paths near the obstacle choose to go through it. We set $r_\epsilon(x,y)=1$ everywhere else, and let the source point be given by $\Gamma=\{(0,0)\}$ . These examples show the performance of the method on problems when the coarse grid captures the flow of characters in the opposite direction. The causal sweeps are critical in order to capture the right flow of characteristics because some coarse grids the right causality will never be computed. The solution contours for $r^3_\epsilon$ and $r^4_\epsilon$ is shown in \cref{fig:a3,fig:b3}, respectively.  

For $r^3_\epsilon$, there are coarse grid points which coincide with the obstacles as well as points in the obstacles that do not coincide with a coarse grid point. The circle barrier in \cref{fig:a3} contains an entire subdomain and the other circle is a fast obstacle that is contained entirely in a subdomain. The non-monotonicity of error is due to the causal sweeps. The method only provides speed up once the right characteristics have been captured around the barriers. This is seen in \cref{fig:c3} where the error starts to decrease monotonically around 20 iterations. For $r^3_\epsilon$, it takes around $2/H$ iterations for the coarse grid to ``see" around the two curved barriers.  

For $r^4_\epsilon$,  the fast obstacle is located in $[0.26,0.27]\times[0,0.6]$. This example demonstrates that the method performs well when there is large collision of characteristics that goes through several subdomains, and there is a fast obstacle affects the characteristics throughout the majority of the domain. i.e., almost every optimal path in \cref{fig:b3} must go through the obstacle.

\begin{figure}[h]
\centering
\subfloat[$r_\epsilon^3$]{\includegraphics{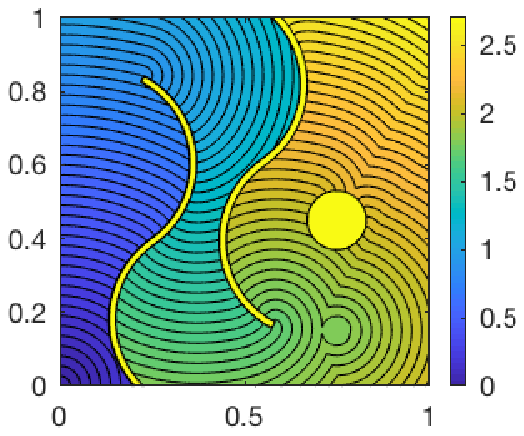}\label{fig:a3}}
\subfloat[$r_\epsilon^4$]{\includegraphics{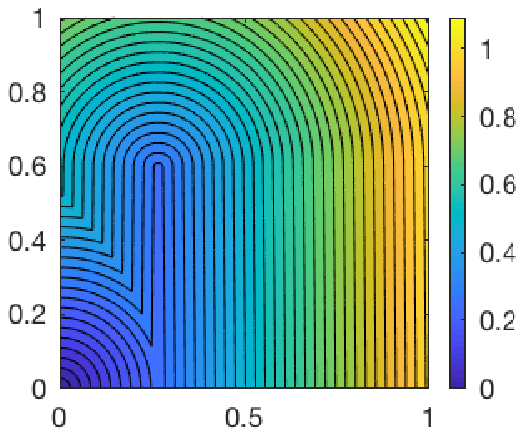}\label{fig:b3}}
\caption{\protect\subref{fig:a3} Solution contour for $r^3_\epsilon$. \protect\subref{fig:b3} Solution contour for $r^4_\epsilon$.}
\label{fig:mazes}
\end{figure}

\begin{figure}[h]
\centering
\includegraphics{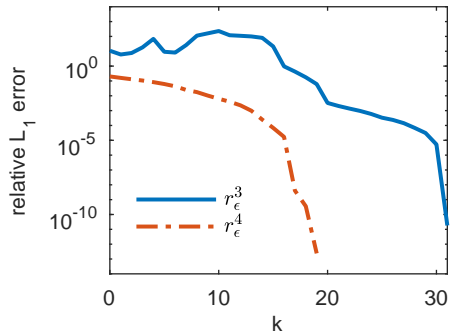}
\caption{Relative $L_1$ error plots for the slowness functions $r_\epsilon^3$ and $r_\epsilon^4$ for $H=1/10$ and $h=1/500$.}
\label{fig:c3}
\end{figure}

\subsection{Multiscale slowness functions}
\begin{figure}[h]
\centering
\subfloat[]{\includegraphics{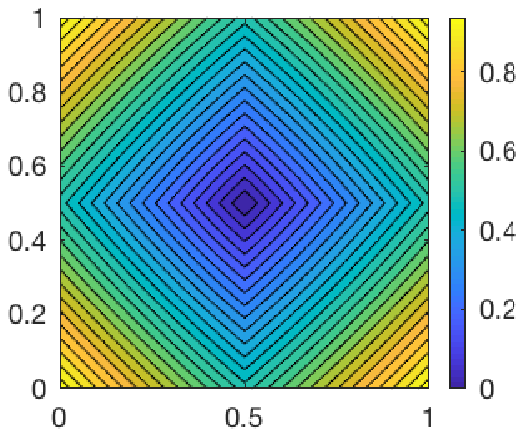}\label{fig:a5}}
\subfloat[]{\includegraphics{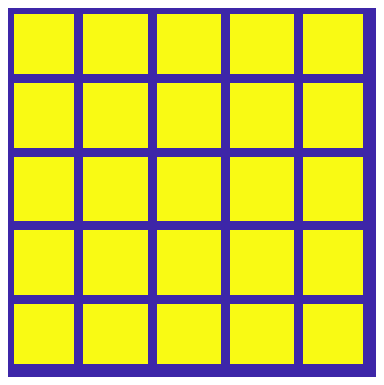}\label{fig:c5}}
\caption{\protect\subref{fig:a5} Solution contour for the squares slowness function where $h=1/1400$ and ${\epsilon=1/200}$. \protect\subref{fig:c5}  Plot of squares slowness function for $\epsilon=1/5$.}
\label{fig:multiscalecontour}
\end{figure}

\begin{figure}[h]
\centering
\includegraphics{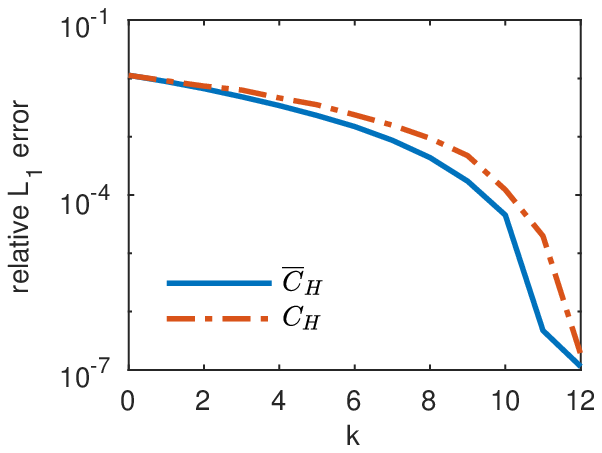}
\caption{Relative error $L_1$ error plots for $C_H$ and $\overline{C}_H$.}
\label{fig:b5}
\end{figure}

 \begin{figure}[h]
\centering
\subfloat[]{\includegraphics{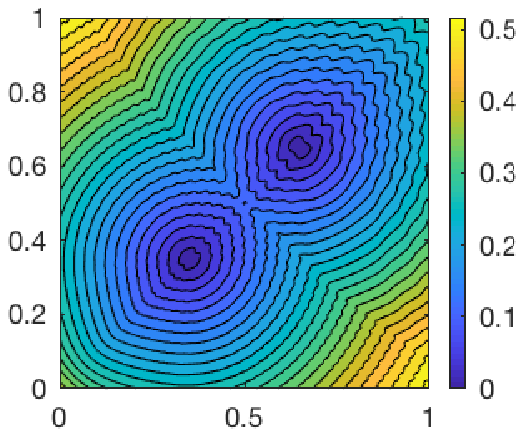}\label{fig:a4}}\;\;\;\;
\subfloat[]{\includegraphics{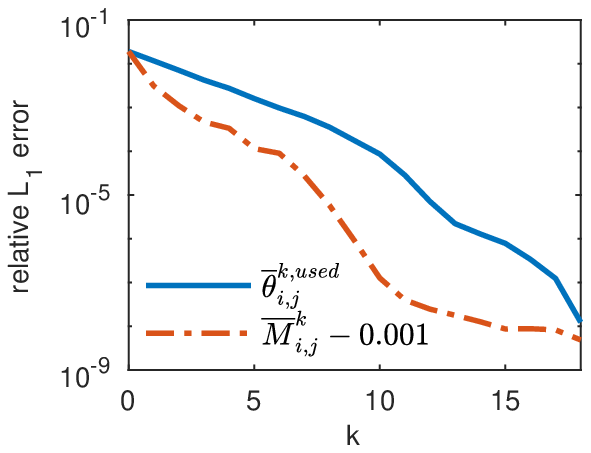}\label{fig:b4}}
\caption{Generalization of \cref{fig:b2}: \protect\subref{fig:a4} Solution contour for ${r_\epsilon(x,y)=1+.5\sin(\frac{\pi x}{\epsilon})\sin(\frac{\pi y}{\epsilon})}$ where ${\epsilon=\frac{|x|+|y|+0.001}{50}}$. \protect\subref{fig:b4} Relative $L_1$ error plot for the given $r_\epsilon$ where the estimated $\overline{\theta}_{i,j}^{k,used}$ and $\overline{M}_{i,j}^k-0.001$ are used.}
\label{fig:generalosccontour}
\end{figure}

Finally, we show the advantage of the method on multiscale slowness functions.  These examples arise in front propagation in multiscale media problems. In our computations, we let the scale epsilon be 7 fine grid points, i.e., $\epsilon=7h,$ in order for the fine grid to fully capture the micro scale behavior. As mentioned in \cref{sec:intro}, one approach for numerically resolving the multiscale behavior in $r_\epsilon$ is homogenization. We will first demonstrate our method on an example where the homogenized slowness function, $\overline{r},$ can be computed.

Let the source point be given by $\Gamma=\{(0.5,0.5)\},$ and define the slowness function as follows: let $$r(x,y)=\begin{cases} 1 & \mbox{ if } x=0 \mbox{ or } y=0 \\ 2 & \mbox{ otherwise }  \end{cases}$$ and define $r_\epsilon$ by extending $r$ by periodicity $\epsilon.$ \cref{fig:c5} shows the slowness function for $\epsilon=1/5$. The homogenized slowness function is anisotropic and is equal to ${\overline{r}(\alpha)=(\alpha_1+\alpha_2)}$ where $\alpha=(\alpha_1,\alpha_2)$ and $|\alpha|=1$. This is due to the optimal paths moving only vertically or horizontally \cite{Ober09}. 

In our computations, we chose $H=1/14,h=1/1400,$ and  $\epsilon=1/200.$ The value of $r_\epsilon$ on the coarse grid points is always equal to 1. Thus, the coarse grid solver is always solving the equation $$|\nabla u|=1.$$ This equation is inaccurate as seen by the shape of the solution contour in \cref{fig:a5} which is a diamond and not a circle.  Suppose in the method we have the coarse solver solve an equation that better describes the macro scale behavior of the solution. Since in this example we know the homogenized equation, on the coarse grid we can solve the homogenized equation \begin{align}\label{eq:homog} \frac{1}{\overline{r}(\frac{\nabla u}{|\nabla u |})}|\nabla \overline{u} |=1.\end{align} Denote the homogenized equation coarse solver by $\overline{C}_H$. \cref{fig:b5} shows the relative $L_1$ error plots for both the method that uses the $C_H$ as described in \cref{sec:main} and the method that uses $\overline{C}_H$ in place of $C_H$. As expected, we can see that method that uses $\overline{C}_H$ performs better.

Next, we demonstrate the method on a generalization of $r^2_\epsilon$ as defined in \cref{sec:smoothslow}. Notice in \cref{fig:b2}, we can see the rough shape of the contours of the solution to the homogenized equation. Let $${r_\epsilon=1+.5\sin(\frac{\pi x}{\epsilon})\sin(\frac{ \pi y}{\epsilon})}.$$ For $r^2_\epsilon$, $\epsilon=1/20$. Now suppose we let $\epsilon$ vary through out the domain, i.e., the problem cannot be solved via homogenization. Define $${\epsilon=\frac{|x|+|y|+0.001}{50}}$$ and ${\Gamma=\{(0.35,0.35),(0.65,0.65)\}}$. Then $\epsilon$ is very small near $(0,0)$ and increases as we move diagonally up and right through the domain. \cref{fig:a4} shows the solution contour for this given $r_\epsilon$. Since $u^f$ can be computed a priori, we compare the error plots of our method where we use formula \cref{eq:thetaest} and $\overline{M}_{i,j}^k-0.001$ as the choice of weights in the method. In this example, $H=1/14$ and $h=1/1400$. The error plots \cref{fig:b4} suggest that the proposed formula \cref{eq:thetaest} has room for improvement in estimating $\overline{M}_{i,j}^k$. For example,   if $X_{i,j}=(0,0)$, we have $${\min_k (|\overline{M}_{i,j}^k-0.001-\overline{\theta}_{i,j}^{k,used}|)=0.1053},$$ but $${\max_k (|\overline{M}_{i,j}^k-0.001-\overline{\theta}_{i,j}^{k,used}|)=2.508\times10^{5}}.$$

Finally, we show the results of our method on a case where the values of the slowness function are random. We follow the set up of the random slowness function in \cite{Ober09}. Consider a periodic checkerboard where the slowness function is either 1 or 2 with probability 1/2. Let the scale of the periodicity be $\epsilon$. A solution contour and the plot of a random slowness function is shown in \cref{fig:a6,fig:c6}, repsectively. In \cite{Ober09}, the authors showed experimentally the homogenized slowness function, $\overline{r}$, is isotropic and its value is a little less than 1.  \cref{fig:b6} shows the plot of the average error over 20 trials where $H=1/14$ and $h=1/1400$.

\begin{figure}[h]
\centering
\subfloat[]{\includegraphics{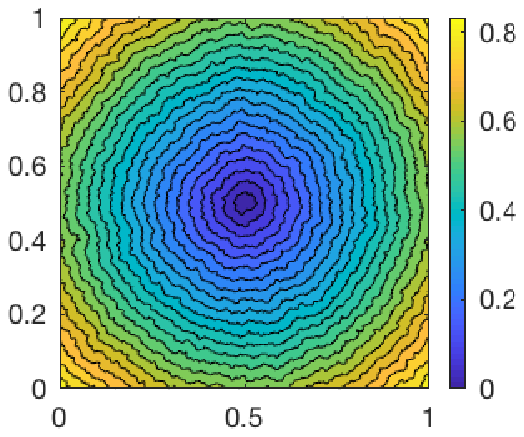}\label{fig:a6}}
\subfloat[]{\includegraphics{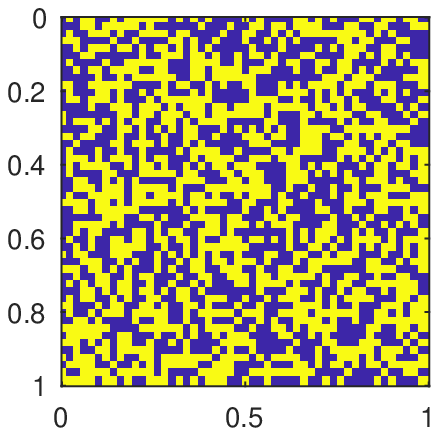}\label{fig:c6}}
\caption{\protect\subref{fig:a6} Solution contour plot of random periodic checkboard of scale $\epsilon$. \protect\subref{fig:c6} Plot of random slowness function.}
\label{fig:randomcontour}
\end{figure}

\begin{figure}[h]
\centering
\includegraphics{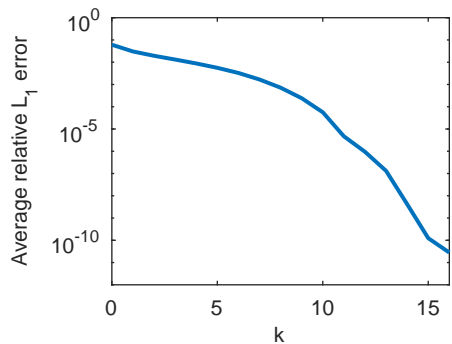}
\caption{Relative $L_1$ error plot for random $r_\epsilon$. }
\label{fig:b6}
\end{figure}

\section{Summary and conclusion}\label{sec:conclusions}
In this paper, we presented a new domain decomposition algorithm for solving boundary value Eikonal equations. Traditional domain decomposition algorithms are difficult to apply due to the nonlinear casual nature of the Eikonal equation. We overcome this difficulty by using coarse and fine grids to propagate information from the subdomains into the coarse level. The parallelization of our method is simple. The coarse grid is initialized using FSM, and the values and wind directions are used to define the boundary conditions for the subdomains. Next, we perform fine grid computations in each subdomain in parallel. In our coarse grid updates we apply an adapted weighted parareal scheme to speed up convergence. A causality sweep is performed after each coarse grid update in order to ensure the wind directions are captured correctly. 

By clever choice of the weight, it is possible to stabilize parareal-like iterative methods. At each coarse grid node, $\theta_{i,j}^k$ is computed by estimating $\overline{M}_{i,j}^k$ which is defined to be the upper bound for $\theta_{i,j}^k$ to create a monotonically decreasing sequence of solutions for the model problem. We show via numerical examples on a model problem that the choice of $\theta_{i,j}^k$ stabilizes the method and using a variable $\theta$ has advantages over a fixed value. We speculate  that improving the estimate of $\overline{M}_{i,j}^k$, it would be possible further increase the stability and speed up of the method.

We demonstrated the method on several classes of slowness functions showing that the performs well on general types of $r_\epsilon$ including multiscale slowness functions where homogenization cannot be applied. The errors decrease to an acceptable tolerance well within the limit of theoretical speed up. Thus, we can solve efficiently through parallelization multiscale problems beyond the conventional multiscale methods. The example in \cref{fig:multiscalecontour} gives us a direction for future work. We would like to use the coarse and fine grid computations to estimate the effective slowness function ``on the fly,'' which could further speed up the method based on the evidence in \cref{fig:b5}.

\bibliographystyle{siamplain}
\bibliography{ms}
\end{document}